\documentclass[mnsc,nonblindrev]{informs3} 

\usepackage{natbib}
 \bibpunct[, ]{(}{)}{,}{a}{}{,}%
 \def\bibfont{\small}%
\TheoremsNumberedThrough     
\ECRepeatTheorems

\EquationsNumberedThrough    

\MANUSCRIPTNO{}

\usepackage{endnotes}
\let\footnote=\endnote

\usepackage[normalem]{ulem}
\newcommand\redsout{\bgroup\markoverwith{\textcolor{red}{\rule[0.5ex]{2pt}{0.4pt}}}\ULon}
\newcommand\gsout{\bgroup\markoverwith{\textcolor{green}{\rule[0.5ex]{2pt}{0.4pt}}}\ULon}

\usepackage{array}
\newcommand{\PreserveBackslash}[1]{\let\temp=\\#1\let\\=\temp}
\newcolumntype{C}[1]{>{\PreserveBackslash\centering}p{#1}}
\newcolumntype{R}[1]{>{\PreserveBackslash\raggedleft}p{#1}}
\newcolumntype{L}[1]{>{\PreserveBackslash\raggedright}p{#1}}

\RequirePackage[shortlabels]{enumitem}
\usepackage{mathtools} 
\usepackage{enumitem} 
\usepackage{dsfont} 
\usepackage{mathrsfs}  
\usepackage{verbatim}
\usepackage{romannum}
\AtBeginDocument{\pagenumbering{arabic}}

\newcommand{\ba}{{\bf a}}
\newcommand{\bb}{{\bf b}}
\newcommand{\bc}{{\bf c}}
\newcommand{\bd}{{\bf d}}

\newcommand{\bg}{{\bf g}}
\newcommand{\bh}{{\bf h}}

\newcommand{\bs}{{\bf s}}

\newcommand{\bv}{{\bf v}}

\newcommand{\bx}{{\bf x}}
\newcommand{\by}{{\bf y}}

\newcommand{\bba}{{\bf A}}
\newcommand{\bbb}{{\bf B}}
\newcommand{\bbc}{{\bf C}}

\newcommand{\bbx}{{\bf X}}
\newcommand{\bby}{{\bf Y}}

\newcommand{\bbz}{{\bf Z}}

\newcommand{\bpi}{{\boldsymbol{\pi}}}

\newcommand{\bgamma}{{\boldsymbol{\gamma}}}

\newcommand{\blambda}{{\boldsymbol{\lambda}}}

\newcommand{\bLambda}{{\boldsymbol{\Lambda}}}

\newcommand{\bzeta}{{\boldsymbol{\zeta}}}

\newcommand{\bxi}{{\boldsymbol{\xi}}}

\newcommand{\F}{\mathcal{F}}    
\newcommand{\Prb}{\mathbb{P}}  
\newcommand{\Exp}{\mathbb{E}}

\newcommand{\I}{\mathds{1}} 

\renewcommand{\P}{\mathbb{P}}
\newcommand{\E}{\mathbb{E}}

\let\N\Natural
\let\Q\Rational
\let\R\Real

\newcommand{\bzero}{{\boldsymbol{0}}}

\usepackage{graphicx}
\usepackage{bbm}
\usepackage{tikz,ifthen}
\usepackage{fixltx2e}    
   \usetikzlibrary{math}
   \usetikzlibrary{shapes.misc}
   \usetikzlibrary{shapes.multipart,positioning,patterns,backgrounds}

\usepackage[singlelinecheck=false]{caption}

\usepackage[ruled,vlined,linesnumbered]{algorithm2e}

\usepackage{multirow}
\usepackage{array}
\usepackage{subcaption}

\def\ie{{\it i.e.}, }
\def\eg{{\it e.g.}, }

\usepackage{hyperref}

\hypersetup{
    colorlinks=true,
    linkcolor=blue,
    filecolor=blue,      
    urlcolor=blue,
    citecolor=blue
}

\begin{document}


 \RUNAUTHOR{Bertsimas, McCord, and Sturt}


\RUNTITLE{Dynamic Optimization with Side Information}

\TITLE{Dynamic optimization with side information}

\ARTICLEAUTHORS{%
\AUTHOR{Dimitris Bertsimas, Christopher McCord, Bradley Sturt}
\AFF{Operations Research Center, Massachusetts Institute of Technology, \\\EMAIL{dbertsim@mit.edu, mccord@mit.edu, bsturt@mit.edu}}
} 

\ABSTRACT{%
We develop a tractable and flexible approach for incorporating side information into dynamic optimization under uncertainty. 
The proposed framework uses predictive machine learning methods (such as $k$-nearest neighbors, kernel regression, and random forests) to weight the relative importance of various data-driven uncertainty sets in a robust optimization formulation. Through a novel measure concentration result for a class of machine learning methods, we prove that the proposed approach is asymptotically optimal for multi-period stochastic programming with side information. We also describe a general-purpose approximation for these optimization problems, based on overlapping linear decision rules, which is computationally tractable and produces high-quality solutions for dynamic problems with many stages.
Across a variety of examples in inventory management, finance, and shipment planning, our method achieves improvements of up to 15\% over alternatives and requires less than one minute of computation time on problems with twelve stages. 
}%


\HISTORY{This paper was first submitted in May 2019. A revision was submitted in May 2020. }

\maketitle

\section{Introduction} \label{sec:intro}
Dynamic decision making under uncertainty forms the foundation for numerous fundamental problems in operations research and  management science. In these problems, a decision maker attempts to minimize an uncertain objective over time, as information incrementally becomes available. For example, consider a retailer with the goal of managing the inventory of a new short life cycle product. Each week, the retailer must decide an ordering quantity to replenish its inventory. Future demand for the product is unknown, but the retailer can base its ordering decisions on the remaining inventory level, which depends on the realized demands in previous weeks. A risk-averse investor faces a similar problem when constructing and adjusting a portfolio of assets in order to achieve a desirable risk-return tradeoff over a horizon of many months. Additional examples abound in energy planning, airline routing, and ride sharing, as well as in many other areas.

To make high quality decisions in dynamic environments, the decision maker must accurately model future uncertainty. Often, practitioners have access to \emph{side information} or \emph{auxiliary covariates}, which can help predict that uncertainty. For a retailer, although the future demand for a newly introduced clothing item is unknown, data on the brand, style, and color of the item, as well as data on market trends and social media, can help predict it. For a risk-averse investor, while the returns of the assets in future stages are uncertain, recent asset returns and prices of relevant options can provide crucial insight into upcoming volatility. Consequently, organizations across many industries are continuing to prioritize the use of predictive analytics in order to leverage vast quantities of data to understand future uncertainty and make better operational decisions.

{\color{black}
In this paper, we address these applications by studying the following class of {multi-period} stochastic decision problems. Specifically, we consider problems faced by organizations in which decisions $\bx_1 \in \mathcal{X}_1 \subseteq \R^{d^1_x},\ldots,\bx_T \in \mathcal{X}_T \subseteq \R^{d^T_x}$ are chosen sequentially, as random vectors  $\bxi_1 \in \Xi_1\subseteq \R^{d^1_\xi},\ldots,\bxi_T \in \Xi_T\subseteq \R^{d^T_\xi}$ become incrementally available at each temporal period. Before selecting any decisions, we observe side information, $\bgamma \in \Gamma \subseteq \R^{d_\gamma}$, which may be predictive of the uncertain quantities observed in the subsequent periods. The goal is to choose a {decision rule} (policy) which minimizes the conditional expected cost over the entire problem horizon:
\begin{align} \label{eq:main}
\begin{aligned}
v^*(\bar{\bgamma}) \triangleq \quad& \underset{\bx_t: \Xi_1 \times \cdots \times \Xi_{t-1} \to \mathcal{X}_t}{\textnormal{minimize}}  && \Exp \left[ c\left(\bxi_1,\ldots,\bxi_T,\bx_1,\bx_2(\bxi_1),\ldots,\bx_T(\bxi_1,\ldots,\bxi_{T-1})\right)  \bigg|\; \bgamma = \bar{\bgamma} \right].
\end{aligned}
\end{align}
However, the only insight into the joint probability distribution $(\bgamma,\bxi_1,\ldots,\bxi_T)$ comes from historical data, $(\bgamma{}^1,\bxi{}_1^1,\ldots,\bxi{}_T^1),\ldots,(\bgamma{}^N,\bxi{}_1^N,\ldots,\bxi{}_T^N)$, which are assumed to be independent and identically distributed (i.i.d.) realizations of the underlying joint distribution. Throughout the paper, we do not impose any parametric structure on the correlations across $(\bgamma,\bxi_1,\ldots,\bxi_T)$, and presume that the structure of optimal decision rules to \eqref{eq:main} is unknown. 
The aim of the present paper is to develop general-purpose approaches to harness this data to approximately solve the stochastic problem \eqref{eq:main}. 

Such dynamic optimization problems with an initial observation of side information arise in many operational contexts. For example, fashion retailers have access to data on the brand, style, and color of a new clothing item prior to any sales, which are predictive of demand for the product in each week of its lifecycle. 
Similarly, in finance, important economic data (such as the consumer price index CPI and key numbers from the US Bureau  of Labor Statistics report) are released monthly on a fixed schedule, and this data serves as side information for a fund manager who seeks to balance the risk of a portfolio in each day of the ensuing month. 
Consequently, from a modeling perspective, \eqref{eq:main} encompasses the variety of decision problems faced by organizations in which side information does not change over time (\eg the fashion retailer) or varies on a much longer time scale than the length of the decision horizon (\eg the fund manager).

A recent body of work has aimed to leverage predictive analytics to address \eqref{eq:main} in the particular case of single-period problems ($T=1$).  
For example, \cite{hannah2010nonparametric,ban2018big,bertsimas2014predictive,hansusanto2019kernel} investigate prescriptive approaches, based on sample average approximation, that use local machine learning to assign weights to the historical data based on side information.  %
 \cite{bertsimas2017bootstrap} propose adding robustness to those weights to achieve optimal asymptotic budget guarantees.   \cite{elmachtoub2017smart} develop an approach for linear optimization problems in which a machine learning model is trained to minimize the decision cost.  
{\color{black}
Unfortunately, prescriptive approaches designed for single-period problems do not generally extend to \eqref{eq:main}, as illustrated by the following example.
\begin{example}
Suppose a decision maker attempted to approximate \eqref{eq:main} by solving 
\begin{align}\label{eq:bad}
\begin{aligned}
& \underset{\bx_t: \Xi_1 \times \cdots \times \Xi_{t-1} \to \mathcal{X}_t}{\textnormal{minimize}}  &&  \sum_{i=1}^N w^i_N(\bar{\bgamma})  c\left(\bxi_1^i,\ldots,\bxi_T^i,\bx_1,\bx_2(\bxi_1^i),\ldots,\bx_T(\bxi_1^i,\ldots,\bxi_{T-1}^i)\right),
\end{aligned}
\end{align}
where $w^i_N(\cdot)$ are weight functions (satisfying $\sum_{i=1}^N w^i_N(\bar{\bgamma}) = 1$) derived from machine learning methods applied to historical data. When $T=1$ and the weight functions are constructed using a suitable class of machine learning methods, \cite{bertsimas2014predictive} show under certain conditions that the above optimization problem is asymptotically optimal, and will thus provide a near-optimal approximation of \eqref{eq:main} in big data settings. 
However, it is readily observed that approaches such as \eqref{eq:bad} will result in a poor approximation of the underlying multi-period stochastic decision problem with side information when $T \ge 2$, as the optimal decision rules produced by \eqref{eq:bad} will generally be ``anticipative" with respect to the historical data.\footnote{If the random vectors are continuous and $T \ge 2$, it is readily observed that \eqref{eq:bad} resolves to an optimization problem of the form
\begin{align*}
\begin{aligned}
& \underset{\bx_1 \in \mathcal{X}_1;\; \bx_2^i \in \mathcal{X}_2,\ldots,\bx_T^i \in \mathcal{X}_T \forall i}{\textnormal{minimize}}  &&  \sum_{i=1}^N w^i_N(\bar{\bgamma})  c\left(\bxi_1^i,\ldots,\bxi_T^i,\bx_1,\bx_2^i,\ldots,\bx_T^i\right).
\end{aligned}
\end{align*}}  Such anticipativity (a form of \emph{overfitting}) is ultimately of practical importance, as it implies that the \eqref{eq:bad} can provide an unsuitable approximation of \eqref{eq:main} even in the presence of big data. 
\end{example}

To circumvent overfitting in the context of multi-period problems with side information, recent literature have aimed to address \eqref{eq:main} by  constructing  scenario trees. Scenario trees have been long studied in the stochastic programming literature, and essentially address overfitting by encoding the various ways that uncertainty can unfold across time. For a class of multi-period inventory management problems with side information, \cite{ban2018dynamic} propose fitting historical data and side information to a parametric regression model, and establish  asymptotic optimality when the model is correctly specified. \citet{bertsimasmccord2018multistage} propose a different approach based on dynamic programming that uses nonparametric machine learning methods to handle auxiliary side information. These papers also extend to problems where side information is observed at multiple periods.  However, these dynamic approaches require scenario tree enumeration and suffer from the curse of dimensionality. As a result, and despite their asymptotic optimality, the existing approaches for addressing \eqref{eq:main} can require hours or days to obtain high-quality solutions for problems with ten or fewer time periods.  

\subsection{Contributions}
The aim of the present paper, in a nutshell, is to develop a machine learning-based approach for addressing \eqref{eq:main} which remains computationally tractable for operational problems with many periods. 
To this end, we develop a new approach to \eqref{eq:main} by a natural combination of prescriptive analytics \eqref{eq:bad} with recent techniques from robust optimization to avoid overfitting  \citep{bertsimas2018multistage},  and the present paper unifies our understanding of these disparate models through  a novel asymptotic theory. 

Our proposed combination of two streams of literature (prescriptive analytics and robust optimization) is ultimately viewed as attractive from a practical standpoint. Across multi-period and single-period problems from several applications (shipment planning, inventory management, and finance), the proposed approach produces solutions with up to 15\% improvement in average out-of-sample cost compared to alternatives. In particular, the approach does not require a scenario tree, and as a result, is significantly more tractable compared to  existing approaches for dynamic optimization with side information. 
To the best of our knowledge, this is the first approach to address \eqref{eq:main} which offers asymptotic optimality guarantees while remaining practically tractable for problems with many periods, thus offering organizations a general-purpose tool for better decision making with predictive analytics.



In greater detail, the key results of this paper are the following: 


\vspace{0.5em}

\begin{enumerate}[(a)]
\item We propose addressing \eqref{eq:main} by combining the prescriptive analytics approach \eqref{eq:bad} with a  technique of \cite{bertsimas2018multistage} to avoid overfitting in multi-period problems.  
\vspace{0.5em}
\item We prove  under mild conditions that this combination of {machine learning} and {robust optimization} is asymptotically optimal for \eqref{eq:main} for general spaces of decision rules (Theorem~\ref{thm:convergence}).

\vspace{0.5em}
\item To establish the above guarantee, we show for the first time that an \emph{empirical conditional probability distribution} that is constructed from machine learning methods will, as more data is obtained, converge to the underlying \emph{conditional probability distribution} with respect to the type-1 Wasserstein distance (Theorem~\ref{thm:conditionalconcentration}). 

\vspace{0.5em}
\item  As a byproduct of the new measure concentration result, we show how side information and machine learning can be tractably incorporated into (single-period) Wasserstein-based distributionally robust optimization problems while maintaining its attractive asymptotic optimality. 

\vspace{0.5em}
\item  To find high quality solutions for problems with many stages in practical computation times, we develop a tractable approximation algorithm for these robust optimization problems  by extending an approach of \cite{bertsimas2019twostage,chen2020robust} to multi-period problems. 

\vspace{0.5em}
\item Across multi-period and single-period problems from several applications (shipment planning, inventory management, and finance), we show that the proposed combination of machine learning and robust optimization outperforms alternatives with up to 15\% improvement in average out-of-sample cost. In particular, the proposed approach is practical and scalable, requiring less than one minute on examples with up to twelve stages.
\end{enumerate}

\vspace{0.5em}

The paper is organized as follows. Section~\ref{sec:prob_setting} introduces the problem setting and notation. Section~\ref{sec:sro} proposes the new framework for incorporating machine learning into dynamic optimization. Section~\ref{sec:theory} develops theoretical guarantees on the proposed approach. Section~\ref{sec:wdro} discusses implications of these results in the context of single-period distributionally robust optimization with the type-1 Wasserstein ambiguity set. Section \ref{sec:approx} presents the general multi-policy approximation scheme for dynamic optimization with side information. Section~\ref{sec:experiments} presents a detailed investigation and computational simulations of the proposed methodology in shipment planning, inventory management, and finance. We conclude in Section~\ref{sec:conclusion}.

\subsection{Comparison to related work}
This paper follows a recent body of literature on data-driven optimization under uncertainty in operations research and management science. Much of this work has focused on the paradigm of distributionally robust optimization, in which the optimal solution is that which performs best in expectation over a worst-case probability distribution from an ambiguity set. Motivated by probabilistic guarantees, distributionally robust optimization has found particular applicability in data-driven settings in which the ambiguity set is constructed using historical data, such as \cite{delage2010distributionally,xu2012distributional,esfahani2015data,van2017data}. 
In particular, the final steps in our convergence result (Section~\ref{sec:main_result}) draw heavily from similar techniques from \cite{esfahani2015data} and \cite{bertsimas2018multistage}. 
 In contrast to previous work, this paper develops a new measure concentration result for the empirical conditional probability distribution (Section~\ref{sec:concentration}) which enables machine learning and side information to be incorporated into sample robust optimization and Wasserstein-based distributionally robust optimization for the first time.

To the best of our knowledge, the proposed combination of machine learning and robust optimization for addressing  \eqref{eq:main} is novel and its theoretical justification does not follow from the existing literature.  With respect to prescriptive analytics, 
 \cite{bertsimas2014predictive} establish asymptotic optimality guarantees for problems of the form \eqref{eq:bad} in the case of $T = 1$. Their result requires that the cost function is {equicontinuous}. Their proof relies on results from the machine learning literature \citep{walk2010strong}, which show that an appropriately constructed \emph{empirical conditional probability distribution} (with weights $\{w_N^i(\bar{\bgamma})\}$ assigned to each historical observation $\bxi^i$) weakly converges to the \emph{true conditional probability distribution} of $\bxi$ given $\bgamma=\bar{\bgamma}$, under certain assumptions. However, the asymptotic optimality and proof techniques do not apply to \eqref{eq:bad} when $T \ge 2$, since the cost function resulting from decision rules is not equicontinuous in general.  For problems without side information, \cite{bertsimas2018multistage} circumvent the requirement of equicontinuity by adding robustness to the historical data. To establish asymptotic optimality, they use the fact that the empirical probability distribution of the uncertainties concentrates around the true distribution with respect to the type-1 Wasserstein distance. In the present paper, we unify these proof techniques by developing a new measure concentration result for machine learning which shows that the  empirical conditional probability distribution produced by appropriate weight functions concentrates around the true conditional probability distribution with respect to the type-1 Wasserstein distance. This establishes the asymptotic optimality of our robustification of \eqref{eq:bad} for multi-stage stochastic decision problems with side information.  

Several recent papers have focused on tractable approximations of two- and multi-stage \emph{distributionally} and \emph{sample} robust optimization. Many approaches are based around policy approximation schemes, including lifted linear decision rules \citep{bertsimas2018adaptive}, $K$-adaptivity \citep{hanasusanto2016k}, and finite adaptability \citep{bertsimas2018multistage}. Alternative approaches include tractable approximations of copositive formulations \citep{natarajan2011mixed,hanasusanto2016conic}. Closest related to the approximation scheme in this paper are \cite{chen2020robust} and \cite{bertsimas2019twostage}, which address two-stage problems via overlapping decision rules. \cite{chen2020robust} propose a \emph{scenario-wise} modeling approach that leads to novel approximations of  various distributionally robust applications, including two-stage distributionally robust optimization using Wasserstein ambiguity sets and expectations of piecewise convex objective functions in single-stage problems.  Independently, \cite{bertsimas2019twostage} investigate a \emph{multi-policy approximation} of two-stage sample robust optimization by optimizing a separate linear decision rule for each uncertainty set and prove that this approximation gap converges to zero as the amount of data goes to infinity.   In Section~\ref{sec:approx} of this paper, we show how to extend similar techniques to dynamic problems with many stages for the first time.


\section{Problem Setting} \label{sec:prob_setting}
As described in the introduction, we consider finite-horizon discrete-time stochastic decision problems. The uncertain quantities observed in each stage are denoted by random variables $\bxi_1 \in \Xi_1\subseteq \R^{d^1_\xi},\ldots,\bxi_T \in \Xi_T\subseteq \R^{d^T_\xi}$, and the decisions made in each stage are denoted by $\bx_1 \in \mathcal{X}_1 \subseteq \R^{d^1_x},\ldots,\bx_T \in \mathcal{X}_T \subseteq \R^{d^T_x}$. Given realizations of the uncertain quantities and decisions, we incur a cost of
\begin{align*}
c\left(\bxi_1,\ldots,\bxi_T,\bx_1,\ldots,\bx_T\right)\in \R.
\end{align*}
Let a decision rule $\bpi = (\bpi_1,\ldots,\bpi_T)$ denote a collection of measurable functions $\bpi_t: \Xi_1 \times \cdots \times \Xi_{t-1} \to \mathcal{X}_t$ which specify what decision to make in stage $t$ based of the information observed up to that point. {\color{black}For notational convenience, let $\Pi$ denote the space of all measurable non-anticipative decision rules.} Given realizations of the uncertain quantities and choice of decision rules, the resulting cost is
\begin{align*}
c^\bpi\left(\bxi_1,\ldots,\bxi_T\right) \triangleq c(\bxi_1,\ldots,\bxi_T,\bpi_1,\bpi_2(\bxi_1),\ldots,\bpi_T(\bxi_1,\ldots,\bxi_{T-1})).
\end{align*}
Before selecting the decision rules, we observe auxiliary side information $\bgamma \in \Gamma \subseteq \R^{d_\gamma}$. For example, in the aforementioned fashion setting, the side information may contain information on the brand, style, and color of a new clothing item and the remaining uncertainties representing the demand for the product in each week of the lifecycle.  

Given a realization of the side information $\bgamma = \bar{\bgamma}$, our goal is to find {decision rules} which minimize the conditional expected cost:\begin{align} \tag{\ref{eq:main}}
\begin{aligned}
v^*(\bar{\bgamma}) \triangleq \quad& \underset{\bpi \in \Pi}{\textnormal{minimize}}  && \Exp \left[ c^\bpi(\bxi_1,\ldots,\bxi_T)\; \bigg|\; \bgamma = \bar{\bgamma} \right].
\end{aligned}
\end{align} 
We refer to \eqref{eq:main} as \emph{dynamic optimization with side information}. The optimization takes place over {a collection $\Pi$ which is any subset of the space of all non-anticipative decision rules.}
In this paper, we assume that the joint distribution of the side information and uncertain quantities $(\bgamma,\bxi_1,\ldots,\bxi_T)$ is unknown, and our knowledge consists of historical data of the form
\begin{equation*}
(\bgamma{}^1,\bxi{}_1^1,\ldots,\bxi{}_T^1),\ldots,(\bgamma{}^N,\bxi{}_1^N,\ldots,\bxi{}_T^N),
\end{equation*}
where each of these tuples consists of a realization of the side information and the following realization of the random variables over the stages. For example, in the aforementioned fashion setting, each tuple corresponds to the side information of a past fashion item as well as its demand over its lifecycle. We will not assume any parametric structure on the relationship between the side information and future uncertainty.

The goal of this paper is a general-purpose, computationally tractable, data-driven approach for approximately solving dynamic optimization with side information.  
In the following sections, we propose and analyze a new framework which leverages nonparametric machine learning, trained from historical data, to predict future uncertainty from side information in a way that leads to near-optimal decision rules to \eqref{eq:main}.

\subsection{Notation}

The joint probability distribution of the side information $\bgamma$ and uncertain quantities $\bxi=(\bxi_1,\ldots,\bxi_T)$ is denoted by $\Prb$. 
For the purpose of proving theorems, we assume throughout this paper that the historical data 
are independent and identically distributed (i.i.d.) samples from this distribution $\Prb$. In other words, we assume that the historical data satisfies
\begin{align*}
((\bgamma{}^1,\bxi{}^1),\ldots,(\bgamma{}^N,\bxi{}^N)) \sim \Prb^N,
\end{align*}
where $\Prb^N \triangleq \Prb \times \cdots \times \Prb$ is the product measure. The set of all probability distributions supported on $\Xi \triangleq \Xi_1\times\cdots\times\Xi_T \subseteq \R^{d_\xi}$ is denoted by $\mathcal{P}(\Xi)$. For each of the side information $\bar{\bgamma} \in \Gamma$, we assume that its conditional probability distribution satisfies $\Prb_{\bar{\bgamma}} \in \mathcal{P}(\Xi)$, where  $\P_{\bar{\bgamma}}(\cdot)$ is shorthand for $\P(\cdot \mid \bgamma = \bar{\bgamma})$. We use ``i.o." as shorthand for ``infinitely often". We sometimes use subscript notation for expectations to specify the underlying probability distribution; for example, the following two expressions are equivalent:
\begin{align*}
\Exp_{\bxi \sim \Prb_{\bar{\bgamma}}} \left[ f(\bxi_1,\ldots,\bxi_T) \right] \equiv \Exp \left[ f(\bxi_1,\ldots,\bxi_T) \mid \bgamma = \bar{\bgamma} \right].
\end{align*}
Finally, we say that the cost function resulting from a policy $\bpi$ is upper semicontinuous if 
\begin{align*}
\limsup_{\bzeta \to \bar{\bzeta}} c^\bpi(\bzeta_1,\ldots,\bzeta_T) \le c^\bpi(\bar{\bzeta}_1,\ldots,\bar{\bzeta}_T)
\end{align*}
for all $\bar{\bzeta} \in \Xi$. 

\section{Sample Robust Optimization with Side Information}\label{sec:sro}
In this section, we present our approach for incorporating machine learning in dynamic optimization. We first review sample robust optimization, and then we introduce the proposed \emph{sample robust optimization with side information} approach to \eqref{eq:main}.

\subsection{Preliminary: sample robust optimization} \label{sec:sro_review}
Consider a stochastic dynamic optimization problem of the form \eqref{eq:main} in which there is no side information. The underlying joint distribution of the random variables $\bxi \equiv (\bxi_1,\ldots,\bxi_T)$ is unknown, but we have data consisting of sample paths, $\bxi^1\equiv(\bxi{}^1_1,\ldots,\bxi{}^1_T),\ldots,\bxi^N\equiv(\bxi{}^N_1,\ldots,\bxi{}^N_T)$. For this setting, 
 sample robust optimization can be used to find approximate solutions in stochastic dynamic optimization. To apply the framework, one constructs an uncertainty set around each sample path in the training data and then chooses the decision rules that optimize the average of the worst-case realizations of the cost. Formally, this framework results in the following robust optimization problem:
\begin{equation} \label{eq:unweightedsro}
\begin{aligned}
&\underset{\bpi \in \Pi}{\textnormal{minimize}}&&  \sum_{i=1}^N  \frac{1}{N} \sup_{\bzeta\in \mathcal{U}^i_N}  c^\bpi(\bzeta_1,\ldots,\bzeta_T),
\end{aligned}
\end{equation}
where $\mathcal{U}_N^i\subseteq\Xi$ is an uncertainty set around $\bxi^i$. Intuitively speaking, \eqref{eq:unweightedsro} chooses the decision rules by averaging over the historical sample paths which are adversarially perturbed. 
Under mild probabilistic assumptions on the underlying joint distribution and appropriately constructed uncertainty sets, \cite{bertsimas2018multistage} show that sample robust optimization converges asymptotically to the underlying stochastic problem and that \eqref{eq:unweightedsro} is amenable to approximations similar to dynamic robust optimization.

\subsection{Incorporating side information into sample robust optimization} \label{sec:sro_with_covariates}
We now present our new framework, based on sample robust optimization, for solving dynamic optimization with side information. In the proposed framework, we first train a machine learning algorithm on the historical data to predict future uncertainty $(\bxi_1,\ldots,\bxi_T)$ as a function of the side information. 
From the trained learner, we obtain weight functions $w_{N}^i(\bar{\bgamma})$, for $i=1,\ldots,N$, each of which captures the relevance of the $i$th training sample to the new side information, $\bar{\bgamma}$. We incorporate the weights into sample robust optimization by multiplying the cost associated with each training example by the corresponding weight function. The resulting \emph{sample robust optimization with side information} framework is as follows:
\begin{equation}\label{eq:sro}
\begin{aligned}
\hat{v}^N(\bar{\bgamma}) \triangleq\quad& \underset{\bpi \in \Pi}{\textnormal{minimize}} && \sum_{i=1}^N w^i_N(\bar{\bgamma}) \sup_{\bzeta \in \mathcal{U}^i_N} c^\bpi(\bzeta_1,\ldots,\bzeta_T),
\end{aligned}
\end{equation}
where the uncertainty sets are defined
\begin{equation*}
\mathcal{U}^i_N \triangleq \left\{\bzeta\in\Xi : \| \bzeta - \bxi{}^i \| \le \epsilon_N\right\},
\end{equation*}
and $\|\cdot\|$ is some $\ell_p$ norm with $p\ge 1$.

The above framework provides the flexibility for the practitioner to construct weights from a variety of machine learning algorithms. We focus in this paper on weight functions which come from nonparametric machine learning methods. Examples of viable predictive models include $k$-nearest neighbors (kNN), kernel regression, classification and regression trees (CART), and random forests (RF).  We describe these four classes of weight functions. 

\begin{definition}
The $k$-nearest neighbor weight functions are given by:
\begin{equation*}
w^i_{N,\text{$k$NN}}(\bar{\bgamma}) \triangleq \begin{cases}
\dfrac{1}{k_N}, & \text{if $\bgamma^i$ is a $k_N$-nearest neighbor of $\bar{\bgamma}$,} \\
0, & \text{otherwise.}
\end{cases}
\end{equation*}
\end{definition}
Formally, $\bgamma^i$ is a $k_N$-nearest neighbor of $\bar{\bgamma}$ if $\left|\left\{j \in\{1,\ldots,N\} \setminus i : \|\bgamma^j - \bar{\bgamma}\| <  \|\bgamma^i - \bar{\bgamma}\|\right\}\right| < k_N$. For more technical details, we refer the reader to \citet{biau2015lectures}.

\begin{definition}
The kernel regression weight functions are given by:
\begin{equation*}
w^i_{N,\text{KR}}(\bar{\bgamma}) \triangleq \frac{K(\|\bgamma^i-\bar{\bgamma}\|/h_N)}{\sum_{j=1}^N K(\|\bgamma^j-\bar{\bgamma}\|/h_N)},
\end{equation*}
where $K(\cdot)$ is the kernel function and $h_N$ is the bandwidth parameter. Examples of kernel functions include the Gaussian kernel, $K(u) = \frac{1}{\sqrt{2\pi}}e^{-u^2/2}$, the triangular kernel, $K(u) = (1-u)\I\{u\le 1\}$, and the Epanechnikov kernel, $K(u) = \frac{3}{4}(1-u^2)\I\{u \le 1\}$. For more information on kernel regression, see \citet[Chapter 6]{friedman2001elements}.
\end{definition}
 The next two types of weight functions we present are based on classification and regression trees \citep{breiman1993cart} and random forests \citep{breiman2001random}. We refer the reader to \citet{bertsimas2014predictive} for technical implementation details.

\begin{definition}
The classification and regression tree weight functions are given by:
\begin{equation*}
w^i_{N,\text{CART}}(\bar{\bgamma}) \triangleq \begin{cases}
\dfrac{1}{|l^N(\bar{\bgamma})|}, & i \in l^N(\bar{\bgamma}), \\
0, & \text{otherwise},
\end{cases}
\end{equation*}
where $l^N(\bar{\bgamma})$ is the set of indices $i$ such that $\bgamma^i$ is contained in the same leaf of the tree as $\bar{\bgamma}$.
\end{definition}

\begin{definition}
The random forest weight functions are given by:
\begin{equation*}
w^i_{N,\text{RF}}(\bar{\bgamma}) \triangleq \frac{1}{B}\sum_{b=1}^B w^{i,b}_{N,\text{CART}}(\bar{\bgamma}),
\end{equation*}
where $B$ is the number of trees in the ensemble, and $w^{i,b}_{N,\text{CART}}(\bar{\bgamma})$ refers to the weight function of the $b$th tree in the ensemble.
\end{definition}
All of the above weight functions come from nonparametric machine learning methods. They are highly effective as predictive methods because they can learn complex relationships between the side information and the response variable without requiring the practitioner to state an explicit parametric form. Similarly, as we prove in Section~\ref{sec:theory}, solutions to \eqref{eq:sro} with these weight functions are asymptotically optimal for \eqref{eq:main} without any parametric restrictions on the relationship between $\bgamma$ and $\bxi$. In other words, incorporating side information into sample robust optimization via \eqref{eq:sro} leads to better decisions asymptotically, even without specific knowledge of how the side information affects the uncertainty.

\section{Asymptotic Optimality} \label{sec:theory}
In this section, we establish asymptotic optimality guarantees for sample robust optimization with side information. We prove that, under mild conditions, (\ref{eq:sro}) converges to (\ref{eq:main}) as the number of training samples goes to infinity. Thus, as the amount of data grows, sample robust optimization with side information becomes an optimal approximation of the underlying stochastic dynamic optimization problem. Crucially, our convergence guarantee does not require parametric restrictions on the space of decision rules (\eg linearity)  or parametric restrictions on the joint distribution of the side information and uncertain quantities. 

\subsection{Main result}
We begin by presenting our main result. The proof of the result depends on some technical assumptions and concepts from distributionally robust optimization.  For simplicity, we defer the statement and discussion of technical assumptions regarding the underlying probability distribution and cost until Sections~\ref{sec:concentration} and \ref{sec:main_result}, and first discuss what is needed to apply the method in practice. The practitioner needs to select a weight function, parameters associated with that weight function, and the radius, $\epsilon_N$, of the uncertainty sets. While these may be selected by cross validation, we show that the method will in general converge if the parameters are selected to satisfy the following: 
\begin{assumption} \label{ass:params}
The weight functions and uncertainty set radius satisfy one of the following:
\begin{enumerate}
\item $\{w_N^i(\cdot)\}$ are $k$-nearest neighbor weight functions with $k_N = \min(\lceil k_3 N^\delta\rceil,N-1)$ for constants $k_3>0$ and $\delta\in(\frac{1}{2},1)$, and $\epsilon_N = \dfrac{k_1}{N^p}$ for constants $k_1 > 0$ and $0 < p <  \min\left(\frac{1-\delta}{d_\gamma},\frac{2\delta-1}{d_\xi+2}\right)$.
\item $\{w_N^i(\cdot)\}$ are kernel regression weight functions with the Gaussian, triangular, or Epanechnikov kernel function and $h_N = k_4 N^{-\delta}$ for constants $k_4>0$ and $\delta \in \left(0,\frac{1}{2d_\gamma}\right)$, and $\epsilon_N = \dfrac{k_1}{N^p}$ for constants $k_1 > 0$ and $0 < p < \min\left(\delta,\frac{1-\delta d_\gamma}{2+d_\xi}\right)$.
\end{enumerate}
\end{assumption}
Given Assumption~\ref{ass:params}, our main result is the following.

\begin{theorem}\label{thm:convergence}
Suppose the weight function and uncertainty sets satisfy Assumption~\ref{ass:params}, the joint probability distribution of $(\bgamma,\bxi)$ satisfies Assumptions \ref{as:subgaussian}-\ref{as:aux} from Section~\ref{sec:concentration}, and the cost function satisfies Assumptions \ref{as:decisionclass_sro}-\ref{as:decisionclass_sp} from Section~\ref{sec:main_result}.
Then, for every $\bar{\bgamma}\in\Gamma$, $$\lim_{N \to \infty} \hat{v}^N(\bar{\bgamma}) = v^*(\bar{\bgamma}), \quad \Prb^\infty\text{-almost surely}.$$ 
\end{theorem}
The theorem says that objective value of (\ref{eq:sro}) 
{\color{black}will }converge almost surely to the optimal value of the full-information problem, (\ref{eq:main}), as $N$ goes to infinity. The assumptions of the theorem require that the joint distribution and the feasible decision rules are well behaved. We will discuss these technical assumptions in more detail in the following sections.

In order to prove the asymptotic optimality of sample robust optimization with side information, we view (\ref{eq:sro}) through the more general lens of Wasserstein-based distributionally robust optimization. We first review some properties of the Wasserstein metric and then prove a key intermediary result, from which our main result follows.

\subsection{Review of the Wasserstein metric}
The Wasserstein metric provides a distance function between probability distributions. In particular, given two probability distributions $\mathbb{Q},\mathbb{Q}' \in \mathcal{P}(\Xi)$, the type-1 Wasserstein distance is defined as the optimal objective value of a minimization problem:
\vspace{0.5em}
\begin{align*}
\mathsf{d}_1 \left(\mathbb{Q},\mathbb{Q}' \right) 
&\triangleq \inf \left \{ \E_{(\bxi,\bxi')\sim\Pi} \left \| \bxi - \bxi' \right \| : \quad \begin{aligned} &\Pi \text{ is a joint distribution of $\bxi$ and $\bxi'$} \\
&\text{with marginals } \mathbb{Q} \text{ and } \mathbb{Q}' \text{, respectively}
\end{aligned} \right \}.\end{align*}
The Wasserstein metric is particularly appealing because a distribution with finite support can have a finite distance to a continuous distribution. This allows us to construct a Wasserstein ball around an empirical distribution that includes continuous distributions, which cannot be done with other popular measures such as the Kullback-Leilbler divergence \citep{kullback1951information}. We remark that the {\color{black}type-}1 Wasserstein metric satisfies the axioms of a metric, including the triangle inequality  \citep{clement2008elementary}: $$\mathsf{d}_1(\Q_1,\Q_2) \le \mathsf{d}_1(\Q_1,\Q_3) + \mathsf{d}_1(\Q_3,\Q_2), \quad \forall \Q_1,\Q_2,\Q_3 \in \mathcal{P}(\Xi).$$ Important to this paper, the {\color{black}type-1} Wasserstein metric admits a dual form, as shown by \citet{kantorovich1958space},
\begin{equation*}
\mathsf{d}_1(\Q,\Q') = \sup_{\text{Lip}(h) \le 1}  \left | \E_{\bxi\sim\Q}[h(\bxi)] - \E_{\bxi\sim\Q'}[h(\bxi)] \right|,
\end{equation*}
where the supremum is taken over all 1-Lipschitz functions. Note that the absolute value is optional in the dual form of the metric, and the space of Lipschitz functions can be restricted to those which satisfy $h(0) = 0$ without loss of generality. 
Finally, we remark that \citet{fournier2015rate} prove under a light-tailed assumption that the 1-Wasserstein distance between the empirical distribution and its underlying distribution concentrates around zero with high probability. Theorem \ref{thm:conditionalconcentration} in the following section extends this concentration result to the setting with side information.

\subsection{Concentration of the empirical conditional probability distribution} \label{sec:concentration}
Given a local predictive method, let the corresponding empirical conditional measure be defined as
\begin{equation*}
\hat{\P}^N_{\bar{\bgamma}} := \sum_{i=1}^N w^i_N(\bar{\bgamma}) \delta_{\bxi{}^i},
\end{equation*}
where $\delta_\bxi$ denotes the Dirac probability distribution which places point mass at $\bxi$. 
In this section, we prove under mild assumptions that the empirical conditional  measure $\hat{\P}^N_{\bar{\bgamma}}$ concentrates quickly to $\P_{\bar{\bgamma}}$ with respect to the $1$-Wasserstein metric. 
We introduce the following assumptions on the underlying joint probability distribution:

\begin{assumption}[Conditional Subgaussianity]\label{as:subgaussian}
There exists a parameter $\sigma > 0$ such that
\begin{equation*}
\P\left(\|\bxi\| - \E[\|\bxi\| \mid \bgamma=\bar{\bgamma}] > t \mid \bgamma=\bar{\bgamma}\right) \le \exp\left(-\frac{t^2}{2\sigma^2}\right) \;\;\; \forall t > 0, \bar{\bgamma}\in\Gamma.
\end{equation*}
\end{assumption}
\begin{assumption}[Lipschitz Continuity]\label{as:wasserstein}
There exists $0 < L < \infty$ such that
\begin{equation*}
\mathsf{d}_1(\P_{\bar{\bgamma}},\P_{\bar{\bgamma}'}) \le L \|\bar{\bgamma}-\bar{\bgamma}'\|, \quad \forall \bar{\bgamma},\bar{\bgamma}'\in\Gamma.
\end{equation*}
\end{assumption}
\begin{assumption}[Smoothness of Side Information]\label{as:aux}
The set $\Gamma$ is compact, and there exists $g > 0$ such that
\begin{align*}
\P(\| \bgamma - \bar{\bgamma} \| \le \epsilon) \ge g\epsilon^{d_\gamma}, \quad \forall \epsilon > 0, \; \bar{\bgamma} \in \Gamma.
\end{align*}
\end{assumption}

Let us reflect on the  conditions on the underlying joint distribution. Assumption \ref{as:subgaussian} requires that the distribution of the uncertainty is not heavy-tailed, conditional on the side information. This is satisfied, for example, if $\bxi$ has bounded support or follows a Gaussian distribution, conditional on $\bar{\bgamma}$. Assumption \ref{as:wasserstein} requires that the conditional distribution of $\bxi$ is a smooth function of $\bgamma$. This ensures we can actually learn about the conditional distribution $\Prb_{\bar{\bgamma}}$ from historical data with side information that are similar (but not identical) to $\bar{\bgamma}$. Assumption \ref{as:aux} ensures the side information are distributed in such a way that every possible $\bar{\bgamma}\in\Gamma$ has nearby observations in the historical data, as $N\to\infty$.

With these assumptions, we are ready to prove the concentration result, which is proved using a novel technique that relies on the dual form of the Wasserstein metric and a discrete approximation of the space of 1-Lipschitz functions.
\begin{theorem}\label{thm:conditionalconcentration}
Suppose the weight function and uncertainty sets satisfy Assumption~\ref{ass:params} and the joint probability distribution of $(\bgamma,\bxi)$ satisfies Assumptions \ref{as:subgaussian}-\ref{as:aux}. Then, for every $\bar{\bgamma}\in\Gamma$,
\begin{equation*}
\P^\infty\left(\left \{\mathsf{d}_1(\P_{\bar{\bgamma}},\hat{\P}_{\bar{\bgamma}}^N) > \epsilon_N \right\} \textnormal{ i.o. } \right) = 0.
\end{equation*}
\end{theorem}
\begin{proof}{Proof.}
Without loss of generality, we assume throughout the proof that all norms $\|\cdot\|$ refer to the $\ell_\infty$ norm.\footnote{To see why this is without loss of generality, consider any other $\ell_p$ norm where $p\ge 1$. In this case, $$\|\bxi-\bxi'\|_p \le d_\xi^{1/p}\|\bxi-\bxi'\|_\infty.$$ By the definition of the 1-Wasserstein metric, this implies $$\mathsf{d}_1^p(\P_{\bar{\bgamma}},\hat{\P}^N_{\bar{\bgamma}}) \le d_\xi^{1/p}\mathsf{d}^\infty_1(\P_{\bar{\bgamma}},\hat{\P}^N_{\bar{\bgamma}}),$$ where $\mathsf{d}_1^p$ refers to the 1-Wasserstein metric with the $\ell_p$ norm. If $\epsilon_N$ satisfies Assumption~\ref{ass:params}, $\epsilon_N/d_\xi^{1/p}$ also satisfies Assumption~\ref{ass:params}, so the result for all other choices of $\ell_p$ norms follows from the result with the $\ell_\infty$ norm.}
Fix any $\bar{\bgamma} \in \Gamma$. It follows from Assumption~\ref{ass:params} that
\begin{align}
&\text{$\{w^i_N(\bar{\bgamma})\}$ are not functions of $\bxi^1,\ldots,\bxi^N$}; \label{line:honesty_assumption}\\
&\sum_{i=1}^N w^i_N(\bar{\bgamma}) = 1 \text{ and } w^1_N(\bar{\bgamma}),\ldots,w^N_N(\bar{\bgamma}) \ge 0,&\forall N \in \N;  \label{line:sum_to_one}\\
&\epsilon_N = \dfrac{k_1}{N^p}, & \forall N \in \N, \label{line:radius}
\end{align}
for constants $k_1,p > 0$. Moreover, Assumption~\ref{ass:params} also implies that there exist constants $k_2 > 0$ and $\eta > p(2+d_\xi)$ such that 
\begin{align}
&\lim_{N \to \infty} \frac{1}{\epsilon_N}\sum_{i=1}^N w_N^i(\bar{\bgamma}) \|\bgamma{}^i-\bar{\bgamma}\| = 0, &\text{  $\P^\infty$-almost surely}; \label{line:lipschitz_convergence} \\
&\E_{\P^N} \left[\exp\left(\frac{-\theta}{\sum_{i=1}^N w_N^i(\bar{\bgamma})^2}\right)\right] \le \exp(-k_2\theta N^\eta), & \forall \theta \in (0,1), N \in \N\label{line:mgf}. 
\end{align}
The proof of the above statements under Assumption~\ref{ass:params} is found in Appendix~\ref{appx:weights}. Now, choose any fixed $q \in (0, \eta/(2+d_\xi) - p)$, and let
\begin{align*}
b_N &\triangleq N^q, &B_N &\triangleq \left \{ \bzeta \in \R^{d_\xi}: \; \| \bzeta \| \le b_N \right \},& I_N \triangleq \I \left \{ \bxi^1,\ldots,\bxi^N \in B_N \right \}.
\end{align*}
Finally, we define the following intermediary probability distributions:
\begin{align*}
\hat{\Q}^N_{\bar{\bgamma}} &\triangleq \sum_{i=1}^N w_N^i(\bar{\bgamma}) \P_{\bgamma{}^{i}}, &\hat{\Q}^N_{\bar{\bgamma} \mid B_N} &\triangleq \sum_{i=1}^N w_N^i(\bar{\bgamma}) \P_{\bgamma^i \mid B_N}, 
\end{align*}
where $\P_{\bgamma^i \mid B_N}(\cdot)$ is shorthand for $\P(\cdot \mid \bgamma = \bgamma^i, \bxi \in B_N)$. 

Applying the triangle inequality for the $1$-Wasserstein metric and the union bound, 
\begin{align*}
\P^\infty\left(\{\mathsf{d}_1(\P_{\bar{\bgamma}},\hat{\P}_{\bar{\bgamma}}^N) > \epsilon_N\} \text{ i.o.}\right) &\le  \P^\infty\left(\left\{ \mathsf{d}_1(\P_{\bar{\bgamma}},\hat{\Q}_{\bar{\bgamma}}^N)  >  \frac{\epsilon_N}{3} \right\} \text{ i.o.}\right)\\
& \quad + \P^\infty\left(\left\{\mathsf{d}_1(\hat{\Q}_{\bar{\bgamma}}^N,\hat{\Q}^N_{\bar{\bgamma} \mid B_N})  > \frac{\epsilon_N}{3} \right\} \text{ i.o.}\right) \\
&\quad+ \P^\infty\left(\left\{\mathsf{d}_1(\hat{\Q}_{\bar{\bgamma} \mid B_N}^N,\hat{\Prb}^N_{\bar{\bgamma}})  >  \frac{\epsilon_N}{3} \right\} \text{ i.o.}\right). 
\end{align*}
We now proceed to bound each of the above terms. 
\subsubsection*{Term 1: $\mathsf{d}_1(\P_{\bar{\bgamma}},\hat{\Q}_{\bar{\bgamma}}^N)$:}
By the dual form of the $1$-Wasserstein metric, 
\begin{align*}
\mathsf{d}_1(\P_{\bar{\bgamma}},\hat{\Q}_{\bar{\bgamma}}^N) &= \sup_{\text{Lip}(h)\le1} \left|\E[h(\bxi)|\bgamma=\bar{\bgamma}] - \sum_{i=1}^N w^i_N(\bar{\bgamma}) \E[h(\bxi)|\bgamma = \bgamma{}^i]\right|,
\end{align*}
where the supremum is taken over all 1-Lipschitz functions. 
By \eqref{line:sum_to_one} and Jensen's inequality, we can upper bound this by
\begin{align*}
\mathsf{d}_1(\P_{\bar{\bgamma}},\hat{\Q}_{\bar{\bgamma}}^N) & \le \sum_{i=1}^N w^i_N(\bar{\bgamma})\left(\sup_{\text{Lip}(h)\le1} \left|\E[h(\bxi)|\bgamma=\bar{\bgamma}] -  \E[h(\bxi)|\bgamma = \bgamma^i]\right|\right)\\
& = \sum_{i=1}^N w_N^i(\bar{\bgamma}) \mathsf{d}_1 \left( \Prb_{\bar{\bgamma}}, \Prb_{\bgamma^i} \right)\\
&\le L\sum_{i=1}^N w_N^i(\bar{\bgamma}) \|\bar{\bgamma}-\bgamma^i\|,
\end{align*}
where the final inequality follows from Assumption \ref{as:wasserstein}. Therefore, it follows from \eqref{line:lipschitz_convergence} that
\begin{align}
 \P^\infty\left( \left\{ \mathsf{d}_1(\P_{\bar{\bgamma}},\hat{\Q}_{\bar{\bgamma}}^N)  > \frac{\epsilon_N}{3} \right\} \text{ i.o.}\right)  = 0. \label{line:first_part_of_first_bound}
\end{align}
\subsubsection*{Term 2: $\mathsf{d}_1(\hat{\Q}^N_{\bar{\bgamma}},\hat{\Q}^N_{\bar{\bgamma} \mid B_N}) $:} Consider any Lipschitz function $\text{Lip}(h) \le 1$ for which $h(0) = 0$, and let $\bar{N} \in \N$ satisfy $b_{\bar{N}} \ge \sigma + \sup_{\bar{\bgamma}\in\Gamma} \E[\|\bxi\||\bgamma=\bar{\bgamma}]$ (which is finite because of Assumption~\ref{as:aux}). Then, for all $N \ge \bar{N}$, and all $\bar{\bgamma}'\in\Gamma$,
\begin{align*}
&\E[h(\bxi)|\bgamma = \bar{\bgamma}'] - \E[h(\bxi)\mid \bgamma = \bar{\bgamma}',\bxi \in B_N] \\
&=\E[h(\bxi)\I\{\bxi \notin B_N\}\mid\bgamma = \bar{\bgamma}'] +\E[h(\bxi)\I\{\bxi \in B_N\}\mid\bgamma = \bar{\bgamma}']  - \E[h(\bxi)\mid \bgamma = \bar{\bgamma}',\bxi \in B_N]\\
&=\E[h(\bxi)\I\{\bxi \notin B_N\}\mid\bgamma = \bar{\bgamma}']  + \E[h(\bxi)\mid\bgamma = \bar{\bgamma}',\bxi \in B_N] \Prb \left(   \bxi \in B_N \mid \bgamma = \bar{\bgamma}' \right)  - \E[h(\bxi)\mid \bgamma = \bar{\bgamma}',\bxi \in B_N]\\
&=\E[h(\bxi)\I\{\bxi \notin B_N\}\mid \bgamma = \bar{\bgamma}']-\E[h(\bxi) \mid \bgamma = \bar{\bgamma}',\bxi \in B_N]\P(\bxi \notin B_N\mid\bgamma = \bar{\bgamma}') \\
&\le \E[ \| \bxi \| \I \{ \bxi \notin B_N \} \mid \bgamma = \bar{\bgamma}'] + b_N \Prb ( \bxi \notin B_N \mid \bgamma = \bar{\bgamma}') \\
&= \int_{b_N}^\infty \Prb \left( \| \bxi \| > t \mid \bgamma = \bar{\bgamma}' \right) dt + b_N \P \left(\| \bxi \| \ge b_N \mid \bgamma = \bar{\bgamma}' \right) \\
&\le (\sigma + b_N) \exp\left(-\frac{1}{2\sigma^2}\left(b_N - \sup_{\bar{\bgamma}'\in\Gamma} \E[\|\bxi\||\bgamma=\bar{\bgamma}']\right)^2\right).
\end{align*}
The first inequality follows because $|h(\bxi)| \le b_N$ for all $\bxi \in B_N$ and  $| h(\bxi) | \le \| \bxi \|$ otherwise. For the second inequality, we used the Gaussian tail inequality $\int_x^\infty e^{-t^2/2}dt \le e^{-x^2/2}$ for $x\ge 1$ \citep{vershynin2018high} along with Assumption \ref{as:subgaussian}.  Because this bound holds uniformly over all $h$, and all $\bar{\bgamma}'\in\Gamma$, it follows that
\begin{align*}
\mathsf{d}_1(\hat{\Q}_{\bar{\bgamma}}^N,\hat{\Q}_{\bar{\bgamma} \mid B_N}^N) &= \sup_{\substack{\text{Lip}(h)\le1, h(0) = 0}} \left|\sum_{i=1}^N w^i_N(\bar{\bgamma})\left(\E[h(\bxi) \mid \bgamma=\bgamma^i] -  \E[h(\bxi) \mid \bgamma = \bgamma^i, \bxi \in B_N]\right)\right| \\
&\le \sum_{i=1}^N w_N^i(\bar{\bgamma}) \sup_{\substack{\text{Lip}(h)\le1, h(0) = 0}} \left|\E[h(\bxi) \mid \bgamma=\bgamma^i] -  \E[h(\bxi) \mid \bgamma = \bgamma^i, \bxi \in B_N]\right| \\
&\le \sup_{\bar{\bgamma}'\in\Gamma}\sup_{\substack{\text{Lip}(h)\le1, h(0) = 0}} \left|\E[h(\bxi) \mid \bgamma=\bar{\bgamma}'] -  \E[h(\bxi) \mid \bgamma = \bar{\bgamma}', \bxi \in B_N]\right| \\
&\le (\sigma + b_N) \exp\left(-\frac{1}{2\sigma^2}\left(b_N - \sup_{\bar{\bgamma}'\in\Gamma} \E[\|\bxi\||\bgamma=\bar{\bgamma}']\right)^2\right),
\end{align*}
for all $N\ge \bar{N}$. It is easy to see that the right hand side above divided by $\epsilon_N/3$ goes to 0 as $N$ goes to infinity, so
\begin{equation*}
 \P^\infty\left( \left\{ \mathsf{d}_1(\hat{\Q}_{\bar{\bgamma}}^N,\hat{\Q}_{\bar{\bgamma} \mid B_N}^N )  > \frac{\epsilon_N}{3} \right\} \text{ i.o.}\right) = 0.
\end{equation*}

\subsubsection*{Term 3: $\mathsf{d}_1(\hat{\Q}^N_{\bar{\bgamma} \mid B_N},\hat{\P}^N_{\bar{\bgamma}}) $:}
By the law of total probability, 
\begin{align*}
\P^N\left(\mathsf{d}_1(\hat{\Q}^N_{\bar{\bgamma} \mid B_N},\hat{\P}^N_{\bar{\bgamma}})  > \frac{\epsilon_N}{3} \right) &\le \P^N(I_N = 0) + \P^N\left(\mathsf{d}_1(\hat{\Q}^N_{\bar{\bgamma} \mid B_N},\hat{\P}^N_{\bar{\bgamma}})  > \frac{\epsilon_N}{3} \bigg| I_N = 1 \right) .
\end{align*}
We now show that each of the above terms have finite summations. First,
\begin{align*}
&\sum_{N=1}^\infty \P^N(I_N = 0) 
 \le \sum_{N=1}^\infty N \sup_{\bar{\bgamma}'\in \Gamma} \P(\bxi \notin B_N \mid \bgamma = \bar{\bgamma}')
 \le \sum_{N=1}^\infty N\sup_{\bar{\bgamma}'\in \Gamma}  \textnormal{exp} \left(- \frac{\left(b_N - \Exp \left[\| \bxi \| \mid \bgamma = \bar{\bgamma}' \right] \right)^2}{2 \sigma^2} \right) < \infty.
\end{align*}
The first inequality follows from the union bound, the second inequality follows from Assumption~\ref{as:subgaussian}, and the final inequality follows because $\sup_{\bar{\bgamma}'\in\Gamma}\E[\|\bxi\| | \bgamma=\bar{\bgamma}'] < \infty$ and the definition of $b_N$. 

Second, for each $l \in \N$, we define several quantities. Let $\mathcal{P}_l$ be the partitioning of $B_N = [-b_N,b_N]^{d_\xi}$ into $2^{ld_\xi }$ translations of $(-b_N2^{-l},b_N2^{-l}]^{d_\xi}$.  Let $\mathcal{H}_l$ be the set of piecewise constant functions which are constant on each region of the partition $\mathcal{P}_l$, taking values on $\{kb_N2^{-l} : k \in \{0,\pm1,\pm2,\pm3,\ldots,\pm2^l\}\}$. Note that $|\mathcal{H}_l| = (2^{l+1}+1)^{2^{ld_\xi}}$.  Then, we observe that for all Lipschitz functions $\text{Lip}(h) \le 1$ which satisfy $h(0)=0$, there exists a $\hat{h} \in \mathcal{H}_l$ such that 
\begin{align*}
\sup\limits_{\bzeta\in B_N}|h(\bzeta) - \hat{h}(\bzeta)| \le b_N 2^{-l+1}.
\end{align*} 
Indeed, within each region of the partition, $h$ can vary by no more than $b_N2^{-l+1}$. The possible function values for $\hat{h}$ are separated by $b_N2^{-l}$. Because $h$ is bounded by $\pm b_N$, this implies the existence of $\hat{h} \in \mathcal{H}_l$ such that $\hat{h}$ has a value within $b_N2^{-l+1}$ of $h$ everywhere within that region. The identical reasoning holds for all other regions of the partition.

Therefore, for every $l \in \N$, 
\begin{align*}
&\P^N\left(\mathsf{d}_1(\hat{\Q}^N_{\bar{\bgamma} \mid B_N},\hat{\P}^N_{\bar{\bgamma}}) > \frac{\epsilon_N}{3} \bigg| I_N = 1\right)\nonumber \\
&= \P^N\left(\sup_{\substack{\text{Lip}(h) \le1\\h(0)=0}} \sum_{i=1}^N w_N^i(\bar{\bgamma})\left(h(\bxi^i) - \E[h(\bxi) \mid \bgamma = \bgamma^i, \bxi\in B_N]\right)  > \frac{\epsilon_N}{3} \Bigg| I_N = 1\right)\nonumber \\
&\le \P^N\left(\sup_{\hat{h} \in \mathcal{H}_l} \sum_{i=1}^N w_N^i(\bar{\bgamma})\left(\hat{h}(\bxi^i) - \E\left[\hat{h}(\bxi)\mid \bgamma = \bgamma^i, \bxi \in B_N \right]\right)  > \frac{\epsilon_N}{3} - 2\cdot b_N 2^{-l+1}\bigg| I_N = 1\right)\nonumber \\
&\le \left| \mathcal{H}_l \right| \sup_{\hat{h} \in \mathcal{H}_l}  \P^N\left(\sum_{i=1}^N w_N^i(\bar{\bgamma})\left(\hat{h}(\bxi^i) - \E\left[\hat{h}(\bxi)\mid \bgamma = \bgamma^i, \bxi \in B_N \right]\right)  > \frac{\epsilon_N}{3} - b_N 2^{-l+2} \bigg| I_N = 1\right),
\end{align*}
where the final inequality follows from the union bound. We choose $l = \left\lceil2 + \log_2 \frac{6b_N}{\epsilon_N}\right\rceil$, in which case
\begin{align*}
\frac{\epsilon_N}{3} - b_N 2^{-l+2} 
\ge \frac{\epsilon_N}{6}.
\end{align*}
Furthermore, for all sufficiently large $N$, 
\begin{align*}
|\mathcal{H}_l| = (2^{l+1}+1)^{2^{ld_\xi }} &\le \left(96\frac{b_N}{\epsilon_N}\right)^{24^{d_\xi}(b_N/\epsilon_N)^{d_\xi}} =  \exp\left(24^{d_\xi}\left(\frac{b_N}{\epsilon_N}\right)^{d_\xi}\log \frac{96b_N}{\epsilon_N}\right).
\end{align*}
Applying Hoeffding's inequality, and noting $|\hat{h}(\bxi^i)|$ is bounded by $b_N$ when $\bxi^i \in B_N$, we have the following for all $\hat{h}\in\mathcal{H}_l$:
\begin{align*}
&\P^N\left( \sum_{i=1}^N w_N^i(\bar{\bgamma})\left(\hat{h}(\bxi^i) - \E[\hat{h}(\bxi)|\bxi\in B_N,\bgamma=\bgamma^i]\right) > \frac{\epsilon_N}{6} \bigg| I_N=1\right) \\
&= \E\left[\P^N\left( \sum_{i=1}^N w_N^i(\bar{\bgamma})\left(\hat{h}(\bxi^i) - \E[\hat{h}(\bxi)|\bxi\in B_N,\bgamma=\bgamma^i]\right) > \frac{\epsilon_N}{6} \bigg| I_N=1,\bgamma^1,\ldots,\bgamma^N\right)\bigg| I_N=1\right] \\
&\le \E\left[\exp\left(-\frac{\epsilon_N^2}{72\sum_{i=1}^N (w_N^i(\bar{\bgamma}))^2 b_N^2}\right)\bigg| I_N=1\right] \\
&= \E\left[\exp\left(-\frac{\epsilon_N^2}{72\sum_{i=1}^N (w_N^i(\bar{\bgamma}))^2 b_N^2}\right)I_N\right]\left(\frac{1}{\P^N(I_N=1)}\right)\\
&\le 2\E\left[\exp\left(-\frac{\epsilon_N^2}{72\sum_{i=1}^N (w_N^i(\bar{\bgamma}))^2 b_N^2}\right)\right]\\
&\le 2\exp\left(-\frac{k_2N^\eta\epsilon_N^2}{72b_N^2}\right),
\end{align*}
for $N$ sufficiently large that $\P(I_N=1) \ge 1/2$ and $\epsilon^2_N/72b^2_N < 1$. Note that \eqref{line:mgf} was used for the final inequality.
Combining these results, we have
\begin{align*}
&\P^N\left(\mathsf{d}_1(\hat{\P}^N_{\bar{\bgamma}},\hat{\Q}^N_{\bar{\bgamma}|B_N}) > \epsilon_N/3 \bigg| I_N=1\right) \le 2\exp\left(24^{d_\xi}\left(\frac{b_N}{\epsilon_N}\right)^{d_\xi}\log \frac{96b_N}{\epsilon_N} - \frac{k_2\epsilon_N^2 N^\eta}{72 Nb_N^2}\right),
\end{align*}
for $N$ sufficiently large. For some constants $c_1,c_2>0$, and sufficiently large $N$, this is upper bounded by
\begin{equation*}
2\exp\left(-c_1N^{\eta-2(p+q)} + c_2 N^{d_{\xi}(q+p)}\log N\right).
\end{equation*}
Since $0 < d_\xi(p+q) < \eta-2(p+q)$, we can conduct a limit comparison test with $1/N^2$ to see that this term has a finite sum over $N$, which completes the proof.
\halmos \end{proof}

\subsection{Proof of main result} \label{sec:main_result}
Theorem~\ref{thm:conditionalconcentration} provides the key ingredient for the proof of the main consistency result.  
We state {\color{black}our final two assumptions on the dynamic optimization problem to establish our main result. 
\begin{assumption}[Regularity of robust problem]\label{as:decisionclass_sro}
For all $\bar{\bgamma} \in\Gamma$, there exists $M \ge 0$ such that the objective value of \eqref{eq:sro} would not change if we restricted its optimization to the decision rules $\bpi \in \Pi$ which satisfy
\begin{align*}
\left| c^\bpi(\bzeta_1,\ldots,\bzeta_T) \right| \le M\left(1 + \max \left \{ \left\| \bzeta \right \|, \sup_{\bzeta' \in \cup_{i=1}^N \mathcal{U}^i_N} \| \bzeta' \| \right \} \right), \quad \forall \bzeta \in \Xi. 
\end{align*}
\end{assumption}
\begin{assumption}[Regularity of stochastic problem]\label{as:decisionclass_sp}
For all $\bar{\bgamma} \in\Gamma$,  there exists $M' \ge 0$ such that the objective value of \eqref{eq:main} would not change if we restricted the optimization to the decision rules $\bpi \in \Pi$ for which $c^\bpi(\bzeta_1,\ldots,\bzeta_T)$ is upper semicontinuous and $\left| c^\bpi(\bzeta_1,\ldots,\bzeta_T) \right| \le M' (1 + \| \bzeta \|)$ for all $\bzeta \in \Xi$. 
\end{assumption}
Assumption~\ref{as:decisionclass_sro} is a minor modification to \citet[Assumption 3]{bertsimas2018multistage}, and can be verified by decision makers through performing a static analysis; see \citet[Appendix A]{bertsimas2018multistage}. Assumption~\ref{as:decisionclass_sp} is a condition on structure of optimal decision rules of the stochastic problem, which is nearly identical to the assumptions of  \citet[Theorem 3.6(i)]{esfahani2015data} which are used to establish asymptotic optimality for distributionally robust optimization with the type-1 Wasserstein ambiguity set. 

 }
Under these assumptions, the proof of Theorem~\ref{thm:convergence} follows from Theorem~\ref{thm:conditionalconcentration} via arguments similar to those used by \citet{esfahani2015data} and \citet{bertsimas2018multistage}.  We state the proof fully in Appendix~\ref{appx:proof_main_theorem}.

\section{Implications for Single-Period Distributionally Robust Optimization} \label{sec:wdro}
Beyond its utility in the context of multi-period problems, the measure concentration result of the previous section (Theorem~\ref{thm:conditionalconcentration}) has potentially valuable implications for distributionally robust optimization with the type-1 Wasserstein ambiguity set. Indeed, consider a  single-period optimization problem of the form
\begin{align}
\begin{aligned}
&\underset{\bx \in \mathcal{X} \subseteq \R^{d_x}}{\text{minimize}}&& \Exp_\Prb \left[ c(\bx,\bxi) \right],\label{prob:single_main}
\end{aligned}
\end{align}
where $\bxi \in \Xi \subseteq \R^{d_\xi}$ is a random vector with a probability distribution $\Prb$. When the distribution is unknown and observable only through limited historical data $(\bxi^1,\ldots,\bxi^N) \sim \Prb^N$, there has been recent interest in approximating the above problems by distributionally robust optimization with the type-1 Wasserstein ambiguity set:
\begin{align} 
\begin{aligned}
&\underset{\bx \in \mathcal{X} \subseteq \R^{d_x}}{\text{minimize}}&& \sup_{\mathbb{Q} \in \mathcal{P}(\Xi): \;  \mathsf{d}_1(\mathbb{Q},\hat{\Prb}^N) \le \epsilon_N }\Exp_\Q \left[ c(\bx,\bxi) \right].  \label{prob:wdro}
\end{aligned}
\end{align}

Due to several attractive  properties, \eqref{prob:wdro} and its relatives have received considerable recent interest in a variety of single-period operational and statistical applications. Indeed, when the robustness parameter is chosen appropriately and other mild assumptions hold, \eqref{prob:wdro} is guaranteed to be asymptotically optimal \citep[Theorem 3.6]{esfahani2015data} and the worst-case cost can often be reformulated as a tractable optimization problem \citep{esfahani2015data,blanchet2019quantifying,gao2016distributionally}. Moreover, there is growing empirical evidence that \eqref{prob:wdro} with a positive choice of the robustness parameter ($\epsilon_N > 0$) can find solutions with significantly better average out-of-sample cost compared to those obtained by the sample average approximation ($\epsilon_N = 0$), particularly when the number of data points is small; see, for example, \citet[Section 7.2]{esfahani2015data} and \citet[Section 4.2]{hanasusanto2016conic}. Theoretical results which aim to explain this improved average out-of-sample cost, both for \eqref{prob:wdro} as well as related robust approaches, are found in \cite{gotoh2018robust} and \cite{anderson2019improving}. 

In the remainder of this section, using the results from Section~\ref{sec:concentration}, we now show how side information and machine learning can be easily incorporated into any problem of the form \eqref{prob:wdro}, without foregoing its asymptotic optimality or computational tractability. Indeed,  consider a  single-period optimization problem of the form
\begin{align}
\begin{aligned}
v^*(\bar{\bgamma}) \triangleq \quad &\underset{\bx \in \mathcal{X} \subseteq \R^{d_x}}{\text{minimize}}&& \Exp_\Prb \left[ c(\bx,\bxi) \mid \bgamma = \bar{\bgamma} \right],\label{prob:single_main_side}
\end{aligned}
\end{align}
where $\bxi \in \Xi \subseteq \R^{d_\xi}$ and $\bgamma \in \Gamma \subseteq \R^{d_\gamma}$ are  random vectors with a joint probability distribution $\Prb$. Assume that the distribution is unknown and observable only through limited historical data $((\bgamma^1,\bxi^1),\ldots,(\bgamma^N,\bxi^N)) \sim \Prb^N$. We address these problems by a modification of distributionally robust optimization with the type-1 Wasserstein ambiguity set, wherein the empirical probability distribution $\hat{\Prb}^N$ is replaced with an empirical conditional probability distribution $\hat{\Prb}^N_{\bar{\bgamma}}$ (see  Section~\ref{sec:concentration}):
\begin{align} 
\begin{aligned}
v^N(\bar{\bgamma}) \triangleq \quad &\underset{\bx \in \mathcal{X} \subseteq \R^{d_x}}{\text{minimize}}&& \sup_{\mathbb{Q} \in \mathcal{P}(\Xi): \;  \mathsf{d}_1(\mathbb{Q},\hat{\Prb}^N_{\bar{\bgamma}}) \le \epsilon_N }\Exp_\Q \left[ c(\bx,\bxi) \right].  \label{prob:wdro_side}
\end{aligned}
\end{align}
As discussed previously, the empirical conditional probability distribution can be constructed using a variety of machine learning methods, such as $k$-nearest neighbor regression or kernel regression. 

For this modification, we obtain the following asymptotic optimality guarantee which is analogous to \eqref{prob:wdro} developed by  \cite{esfahani2015data}.
\begin{theorem} \label{thm:esfahani}
Suppose the weight function and uncertainty sets satisfy Assumption~\ref{ass:params} and the joint probability distribution of $(\bgamma,\bxi)$ satisfies Assumptions \ref{as:subgaussian}-\ref{as:aux}. Assume that $\hat{\bx}_N$ represents an optimizer of \eqref{prob:wdro_side}. Then the following hold for every $\bar{\bgamma}\in\Gamma$:
\begin{enumerate}[(i)]
\item If $c(\bx,\bxi)$ is upper semicontinuous in $\bxi$ and there exists $L \ge 0$ with $| c(\bx,\bxi)| \le L(1+\| \bxi\|)$ for all $\bx \in \mathcal{X}$ and $\bxi \in \Xi$, then $\Prb^\infty$-almost surely we have $\hat{v}^N(\bar{\bgamma}) \downarrow v^*(\bar{\bgamma})$ as $N \to \infty$. 
\item If the assumptions of assertion (i) hold, $\mathcal{X}$ is closed, and $c(\bx,\bxi)$ is lower semicontinuous in $\bx$ for every $\bxi \in \Xi$, then any accumulation point of $\{\hat{\bx}_N\}_{N \in \N}$  is $\Prb^\infty$-almost surely an optimal solution for \eqref{prob:single_main_side}. 
\end{enumerate}
\end{theorem}
\begin{proof}{Proof.}
The proof follows from identical reasoning as \citet[Theorem 3.6]{esfahani2015data}, in which the measure concentration result of  \citet[Theorem 2]{fournier2015rate} is replaced by Theorem~\ref{thm:conditionalconcentration} of the present paper. \halmos
\end{proof}

From the perspective of computational tractability, it is readily observed that \eqref{prob:wdro_side} retains an identical computational tractability as \eqref{prob:wdro}, except where terms of the form $\frac{1}{N}$ are replaced with $w^i_N(\bar{\bgamma})$; see, for example, \citet[Theorem 4.2]{esfahani2015data}. 
As a result of Theorem \ref{thm:esfahani}, we conclude that side information can be tractably incorporated into the variety of operational applications that utilize (single-period) Wasserstein-based distributionally robust optimization. 


\section{Tractable Approximations}\label{sec:approx}
In the previous sections, we presented the new framework of sample robust optimization with side information and established its asymptotic optimality in the context of \eqref{eq:main} without any significant structural restrictions on the space of decision rules. In this section, we focus on tractable methods for approximately solving the robust optimization problems that result from this proposed framework. Specifically, we develop a formulation which uses auxiliary decision rules to approximate the cost function.   
In combination with linear decision rules, this approach enables us to find high-quality decisions for real-world problems with more than ten stages in less than one minute, as we demonstrate in Section~\ref{sec:experiments}. 

We focus in this section on dynamic optimization problems with cost functions of the form
\begin{equation}\label{eq:costfunction}
\begin{aligned}
&c\left(\bxi_1,\ldots,\bxi_T,\bx_1,\ldots,\bx_T\right) \\
&\quad= \sum_{t=1}^T \left( {\bf f}_t^\intercal \bx_t + \bg_t^\intercal \bxi_t +   \min_{\by_t \in \R^{d_y^t}} \left \{\bh^\intercal_t \by_t: \;  \sum_{s=1}^t \bba_{t,s} \bx_s + \sum_{s=1}^t \bbb_{t,s} \bxi_s + \bbc_t \by_t \le \bd_t  \right \}\right).
\end{aligned}
\end{equation}
Such cost functions appear frequently in applications such as inventory management and supply chain networks. 
Unfortunately, it is well known that these cost functions are convex in the uncertainty $\bxi_1,\ldots,\bxi_T$. Thus, even {evaluating} the worst-case cost over a convex uncertainty set is computationally demanding in general, as it requires the maximization of a convex function. 

As an intermediary step towards developing an approximation scheme for \eqref{eq:sro} with the above cost function, we consider the following optimization problem:
\begin{equation} \label{prob:multi_policy}
\begin{aligned}
\tilde{v}^N(\bar{\bgamma}) \triangleq \quad&\underset{\bpi \in \Pi, \; \by_t^i \in \mathcal{R}_t\; \forall i,t}{\textnormal{minimize}} && \sum_{i=1}^N w^i_N(\bar{\bgamma}) \sup_{\bzeta \in \mathcal{U}^i_N} \sum_{t=1}^T \left({\bf f}_t^\intercal \bpi_t(\bzeta_1,\ldots,\bzeta_{t-1}) + \bg_t^\intercal \bzeta_t + \bh_t^\intercal \by_t^i(\bzeta_1,\ldots,\bzeta_{t}) \right)\\
&\textnormal{subject to}&&  \sum_{s=1}^t \bba_{t,s} \bpi_s(\bzeta_1,\ldots,\bzeta_{s-1}) + \sum_{s=1}^t \bbb_{t,s} \bzeta_s + \bbc_t \by_t^i(\bzeta_1,\ldots,\bzeta_t) \le \bd_t  \\
&&&\quad \forall \bzeta \in \mathcal{U}_N^i, \;i \in \{1,\ldots,N\}, \; t \in \{1,\ldots,T\},
\end{aligned}
\end{equation}
where $\mathcal{R}_t$ is the set of all functions $\by: \Xi_1 \times \cdots \times \Xi_t \to \R^{d_y^t}$.  In this problem, we have introduced auxiliary decision rules which capture the  minimization portion of  \eqref{eq:costfunction} in each stage. 
We refer to \eqref{prob:multi_policy} as a \emph{multi-policy} approach, as it involves different auxiliary decision rules for each uncertainty set. 
The following theorem shows that \eqref{prob:multi_policy} is equivalent to \eqref{eq:sro}.
\begin{theorem} \label{thm:multi_policy}
For cost functions of the form \eqref{eq:costfunction}, $\tilde{v}^N(\bar{\bgamma}) = \hat{v}^N(\bar{\bgamma})$. 
\end{theorem}
\begin{proof}{Proof.}
See Appendix~\ref{appx:approx}. \Halmos
\end{proof}
We observe that \eqref{prob:multi_policy} involves optimizing over decision rules, and thus is computationally challenging to solve in general. Nonetheless, we can obtain a tractable approximation of \eqref{prob:multi_policy} by further restricting the space of primary and auxiliary decision rules. 
For instance, we can restrict all primary and auxiliary decision rules as linear decision rules of the form
\begin{align*} 
\bpi_t \left( \bzeta_1,\ldots,\bzeta_{t-1} \right) &= \bx_{t,0} + \sum_{s=1}^{t-1} \bbx_{t,s} \bzeta_s,& \by_t^i(\bzeta_1,\ldots,\bzeta_t)  &=\by_{t,0}^i + \sum_{s=1}^{t} \bby_{t,s}^i \bzeta_s.
\end{align*} 
One can alternatively elect to use a richer class of decision rules, such as lifted linear decision rules \citep{Chen2009,Georghiou2015}. In all cases, feasible approximations that restrict the space of decision rules of \eqref{prob:multi_policy} provide an upper bound on the cost $\hat{v}^N(\bar{\bgamma})$ and produce decision rules that are feasible for \eqref{prob:multi_policy}.

The key benefit of the multi-policy approximation scheme is that it offers many degrees of freedom in approximating the nonlinear cost function. Specifically, in \eqref{prob:multi_policy}, a separate auxiliary decision rule $\by_t^i$ captures the value of the cost function for each uncertainty set in each stage. We approximate each $\by_t^i$ with a linear decision rule, which only needs to be locally accurate, \ie accurate for realizations in the corresponding uncertainty set. As a result, \eqref{prob:multi_policy} with linear decision rules results in significantly tighter approximations of \eqref{eq:sro} compared to using a single linear decision rule, $\by_t$, for all uncertainty sets in each stage.  Moreover, these additional degrees of freedom come with only a mild increase in computation cost, and we substantiate these claims via computational experiments in Section~\ref{sec:dynamicprocurement}. In Appendix~\ref{appx:reformulation}, we provide the reformulation of the multi-policy approximation scheme with linear decision rules into a deterministic optimization problem using standard techniques from robust optimization.

\section{Computational Experiments} \label{sec:experiments}
We perform computational experiments to assess the out-of-sample performance and computational tractability of the proposed methodologies across several applications. These examples are   dynamic inventory management (Section~\ref{sec:dynamicprocurement}), portfolio optimization (Section~\ref{sec:portfolio}), and shipment planning (Section~\ref{sec:shipmentplanning}). 

\begin{table}[t]
\TABLE{Relationship of four methods. \label{tbl:comparison}}
{\centering
\begin{tabular}{lcc}
\hline\\[-2ex]
$\epsilon_N$&$w^i_N(\bar{\bgamma}) = \frac{1}{N}$ for all $i$& $w^i_N(\bar{\bgamma})$ from machine learning\\[0.5ex]
  \hline \\[-2ex]
  $=0$ & Sample average approximation & \citet{bertsimas2014predictive} \\[0.5ex]
  $>0$ & \cite{bertsimas2018multistage} & This paper \\[0.5ex]
  \hline
\end{tabular}}
{}
\end{table}

We compare several methods using different machine learning models. These methods include the proposed {sample robust optimization with side information}, sample average approximation (SAA),  the predictions to prescriptions (PtP) approach of \citet{bertsimas2014predictive}, and sample robust optimization without side information (SRO). 
In Table~\ref{tbl:comparison}, we show that each of the above methods are particular instances of  \eqref{eq:sro} from Section~\ref{sec:sro}. The methods in the left column ignore side information by assigning equal weights to each uncertainty set, and the methods in the right column 
incorporate side information by choosing the weights based on predictive machine learning. The methods in the top row 
do not incorporate any robustness ($\epsilon_N = 0$), and the methods in the bottom row 
incorporate robustness via a positive choice of the robustness parameter ($\epsilon_N>0$) in the uncertainty sets. In addition, for the dynamic inventory management example, we also implement and compare to the residual tree algorithm described in \cite{ban2018dynamic}. In each experiment, the relevant methods are applied to the same training datasets, and their solutions are evaluated against a common testing dataset. Further details are provided in each of the following sections. 



 \subsection{Dynamic inventory management}\label{sec:dynamicprocurement}
We first consider a dynamic inventory control problem over the first $T=12$ weeks of a new product. In each week, a retailer observes demand for the product and can replenish inventory via procurement orders to different suppliers with lead times. Our problem setting closely follows \cite{ban2018dynamic}, motivated by the fashion industry in which retailers have access to auxiliary side information on the new product (color, brand) which are predictive of how demand unfolds over time.  

\paragraph{Problem Description.} In each stage $t \in \{1,\ldots,T\}$, the retailer procures inventory from multiple suppliers to satisfy demand for a single product.  The demands for the product across stages are denoted by $\xi_1,\ldots,\xi_T \ge 0$.  In each stage $t$, and before the demand $\xi_t$ is observed, the retailer places procurement orders at various  suppliers indexed by $\mathcal{J} = \{1,\ldots, | \mathcal{J}| \}$. Each supplier $j \in \mathcal{J}$ has per-unit order cost of $c_{tj} \ge 0$ and a lead time of $\ell_j$ stages. At the end of each stage, the firm incurs a per-unit holding cost of $h_t$ and a backorder cost of $b_t$. Inventory is fully backlogged and the firm starts with zero initial inventory. The cost incurred by the firm over the time horizon is captured by
 \begin{align*}
  \begin{aligned}
c(\xi_1,\ldots,\xi_T,\bx_1,\ldots,\bx_T) = \sum_{t=1}^T \sum_{j \in \mathcal{J}} c_{tj} x_{tj} +\sum_{t=1}^T
\quad& \underset{y_t \in \R}{\textnormal{minimize}} &&y_t \\
&\textnormal{subject to} &&y_t \ge h_t \left(\sum_{j \in \mathcal{J}} \sum_{s=1}^{t-\ell_j}x_{sj} - \sum_{s=1}^t \xi_s\right)\\
&&&y_t \ge -b_t \left(\sum_{j \in \mathcal{J}} \sum_{s=1}^{t-\ell_j}x_{sj} - \sum_{s=1}^t \xi_s \right).
 \end{aligned}  
 \end{align*}
\paragraph{Experiments.}  
The parameters of the procurement problem were chosen based on  \cite{ban2018dynamic}. Specifically, we consider the case of two suppliers where $c_{t1} = 1.0$, $c_{t2} = 0.5$, $h_t = 0.25$, and $b_t = 11$ for each stage. The first supplier has no lead time and the second supplier has a lead time of one stage. We generate training and test data from the same distribution as a shipment planning problem of \citet[Section EC.6]{bertsimas2014predictive}, with the exception that we generate the side information as i.i.d. samples as opposed to an ARMA process (but with the same marginal distribution).   In this case, the demands produced by this data generating process are interpreted as the demands over the $T=12$ stages. We perform computational experiments comparing the proposed sample robust optimization with side information and the residual tree algorithm proposed by \citet{ban2018dynamic}. In particular, we compare sample robust optimization with side information \emph{with} the multi-policy approximation as well as \emph{without} the multi-policy approximation (in which we use a single auxiliary linear decision rule for $y_t$ for all uncertainty sets in each stage).  The uncertainty sets from Section~\ref{sec:sro} are defined  with the $\ell_2$ norm and $\Xi = \R^{12}_+$. 
 The out-of-sample cost resulting from the decision rules were averaged over $100$ training sets of size $N = 40$ and $100$ testing points, and sample robust optimization with side information used $k$-nearest neighbors with varying choices of $k$ and radius $\epsilon \ge 0$ of the uncertainty sets. 
 \begin{table}[t]
\TABLE{Average out-of-sample cost for dynamic inventory problem. \label{tbl:dynamic_procurement}}
{\centering
\begin{tabular}{lccccccccc}
\hline
&& \multicolumn{8}{c}{$\epsilon$}\\
  \cline{3-10}	 \\[-2.5ex]
Method&$k$& 0 & 100 & 200 & 300 & 400 & 500 & 600 & 700 \\ 
  \hline
Sample robust optimization&\\
\quad \quad Linear decision rules \\
\quad \quad \quad \quad no side information & & 9669 & 8783 & 8590 & 8789 & 9150 & 9604 & 10102 & 10614 \\ 
\quad \quad \quad \quad k-nearest neighbors & 26 & 9600 & 8566 & 8411 & 8642 & 9030 & 9494 & 10001 & 10528 \\ 
 & 20 & 9640 & 8544 & 8375 & 8603 & 8996 & 9464 & 9974 & 10505 \\ 
&13 & 9862 & 8561 & 8365 & 8573 & 8960 & 9433 & 9943 & 10473 \\ 
\quad \quad Linear decision rules with multi-policy \\
\quad \quad\quad \quad  no side information&  & 8967 & 7759 & 7360 & 7320 & 7460 & 7716 & 8038 & 8412 \\ 
\quad \quad\quad \quad  k-nearest neighbors& 26 & 11346 & 8728 & 7651 & 7269 & 7241 & 7381 & 7636 & 7966 \\ 
 & 20 & 13012 & 9460 & 7925 & 7328 & \textbf{7195} & 7289 & 7519 & 7835 \\ 
&13 & 16288 & 10975 & 8576 & 7585 & 7243 & 7236 & 7412 & 7697 \\ 
    \hline
\end{tabular}}
{Average out-of-sample cost for the dynamic inventory problem using sample robust optimization with $N = 40$. For each uncertainty set radius $\epsilon$ and parameter $k$,  average was taken over 100 training sets and 100 test points. Optimal is indicated in bold.  
The residual tree algorithm with a binning of $B=2$ in each stage gave an average out-of-sample cost of $27142$.}
\end{table}
\begin{table}[t]
\TABLE{Statistical significance for dynamic inventory problem. \label{tbl:dynamic_procurementpvalue}}
{\centering
\begin{tabular}{lccccccccc}
\hline
&& \multicolumn{8}{c}{$\epsilon$}\\
  \cline{3-10}	 \\[-2.5ex]
Method&$k$& 0 & 100 & 200 & 300 & 400 & 500 & 600 & 700 \\ 
  \hline
Sample robust optimization&\\
\quad  Linear decision rules \\
\quad   \quad no side information & & *& * & *& * & * & * & * & *  \\ 
\quad   \quad k-nearest neighbors & 26 & *& * & *& * & * & * & * & *  \\ 
 & 20 & *& * & *& * & * & * & * & *  \\ 
 &13& *& * & *& * & * & * & * & *  \\ 
\quad  Linear decision rules with multi-policy \quad \quad \quad \\
\quad \quad   no side information&  &* & * &* & * &* & * & * & *  \\ 
\quad \quad   k-nearest neighbors& 26 & *& * & *& * & $ * $ & * & * & *  \\ 
 &20 & *&* & * & * &  -  & * & * & * \\ 
&13 & * & * & * & * & $5.8 \times 10^{-3}$ & $1 \times 10^{-3}$ & * & * \\ 
    \hline
\end{tabular}}
{The $p$-values of the Wilcoxon signed rank test for comparison with sample robust optimization using linear decision rules with multi-policy, $k=20$, and $\epsilon = 400$. An asterisk denotes that the $p$-value was less than $10^{-8}$. After adjusting for multiple hypothesis testing, each result is significant at the $\alpha = 0.05$ significance level if its $p$-value is less than $\frac{0.05}{63} \approx 7.9 \times 10^{-4}$.   
}
\end{table}
\begin{table}[t]
\TABLE{Average computation time (seconds) for dynamic inventory problem. \label{tbl:dynamic_procurement_times}}
{\centering
\begin{tabular}{lccccccccc}
\hline
&& \multicolumn{8}{c}{$\epsilon$}\\
  \cline{3-10}	 \\[-2.5ex]
Method&$k$& 0 & 100 & 200 & 300 & 400 & 500 & 600 & 700 \\ 
  \hline
Sample robust optimization&\\
\quad  Linear decision rules \\
\quad   \quad no side information &  & 3.86 & 25.04 & 24.75 & 25.82 & 28.70 & 35.37 & 31.13 & 31.95 \\
\quad   \quad k-nearest neighbors & 26 & 4.02 & 25.43 & 23.39 & 25.15 & 27.88 & 33.42 & 30.87 & 31.60 \\ 
 & 20 & 3.99 & 25.98 & 23.56 & 24.93 & 27.41 & 32.67 & 30.69 & 31.50 \\ 
 &13 & 4.19 & 26.53 & 24.89 & 24.99 & 26.79 & 31.64 & 30.23 & 31.32 \\ 
\quad  Linear decision rules with multi-policy\\
\quad \quad   no side information&  & 0.16 & 28.31 & 30.01 & 29.05 & 31.13 & 36.03 & 35.57 & 36.09 \\ 
\quad \quad   k-nearest neighbors& 26& 0.15 & 27.74 & 28.69 & 27.78 & 30.54 & 34.44 & 35.50 & 36.15 \\  
 &20 & 0.15 & 27.87 & 28.51 & 27.74 & 30.60 & 34.36 & 35.65 & 36.99 \\  
&13 & 0.14 & 27.78 & 28.30 & 27.27 & 30.00 & 33.67 & 35.91 & 37.76 \\ 
    \hline
\end{tabular}}
{Average computation time (seconds) for the dynamic inventory problem using sample robust optimization with $N = 40$. For each choice of uncertainty set radius $\epsilon$ and parameter $k$,  average was taken over 100 training sets. 
The residual tree algorithm of \citet{ban2018dynamic} with a binning of $B=2$ in each stage had an average computation time of $23.20$ seconds. We were unable to run this algorithm with binning of $B=3$ in each stage.
}
\end{table}

\paragraph{Results.} In Table~\ref{tbl:dynamic_procurement}, we show the average out-of-sample cost resulting from sample robust optimization with side information using linear decision rules, with and without the multi-policy approximation from Section~\ref{sec:approx}. In both settings, we used $k$-nearest neighbors as the machine learning method and evaluated the out-of-sample performance by applying the linear decision rules for the ordering quantities. 
The results of these computational experiments in Table~\ref{tbl:dynamic_procurement} demonstrate that significant improvements in average out-of-sample performance are found when combining the multi-policy approximation with side information via $k$-nearest neighbors. 
We show in Table~\ref{tbl:dynamic_procurementpvalue} that these results are statistically significant. 
For comparison, we also implemented the residual tree algorithm from \citet{ban2018dynamic}. When using their algorithm with a binning of $B=2$ in each stage, their approach resulted in an average out-of-sample cost of $27142$. We were unable to run with a binning of $B=3$ in each stage due to time limitations of $10^3$ seconds, as the size of the resulting linear optimization problem scales on the order $O(B^T)$. Such results are consistent with the estimations of computation times presented in \citep[Section 6.3]{ban2018dynamic}. The running times of the various methods are displayed in Table~\ref{tbl:dynamic_procurement_times}. 

\subsection{Portfolio optimization} \label{sec:portfolio}
The guarantees developed in this paper (Theorem~\ref{thm:convergence}) and the above numerical experiment shows that \eqref{eq:sro} is practically tractable and performs well in problems where $T \ge 1$. 
In the current and the following section, we provide numerical evidence that \eqref{eq:sro} can also outperform existing approaches on single-period problems.  

Specifically, in this section we consider a single-stage portfolio optimization problem in which we wish to find an allocation of a fixed budget to $n$ assets. Our goal is to simultaneously maximize the expected return while minimizing the the conditional value at risk (cVaR) of the portfolio. Before selecting our portfolio, we observe auxiliary side information which include general market indicators such as index performance as well as macroeconomic numbers released by the US Bureau of Labor Statistics.


\paragraph{Problem Description.} We denote the portfolio allocation among the assets by $\bx \in \mathcal{X} \triangleq \{ \bx \in \R^n_+: \sum_{j=1}^n x_j = 1 \}$, and the returns of the assets by the random variables $\bxi\in\R^{n}$.  The conditional value at risk at the $\alpha \in (0,1)$ level measures the expected loss of the portfolio, conditional on losses being above the $1-\alpha$ quantile of the loss distribution. \citet{rockafellar2000optimization} showed that the cVaR of a portfolio can be computed as the optimal objective value of a convex minimization problem. Therefore, our portfolio optimization problem can be expressed as a convex optimization problem with an auxiliary decision variable, $\beta\in\R$. Thus, given an observation $\bar{\bgamma}$ of the auxiliary side information, our goal is to solve
\begin{equation}
\begin{aligned}
&\underset{\bx \in \mathcal{X}, \;\beta \in \R}{\textnormal{minimize}}&& \Exp \left[\beta + \frac{1}{\alpha}\max(0,-\bx^\intercal \bxi-\beta) - \lambda \bx^\intercal \bxi \; \bigg| \;\bgamma = \bar{\bgamma} \right], \label{prob:portfolio}
\end{aligned}
\end{equation}
where $\lambda\in\R_+$ is a trade-off parameter that balances the risk and return objectives. 

\paragraph{Experiment.} 
Our experiments are based on a similar setting from  \citet[Section 5.2]{bertsimas2017bootstrap}. Specifically, we perform computational experiments on an instance with parameters $\alpha = 0.05$ and $\lambda = 1$, and the joint distribution of the side information and asset returns are chosen the same as \citet[Section 5.2]{bertsimas2017bootstrap}. In our experiments, we compare sample robust optimization with side information, sample average approximation, sample robust optimization, and predictions to prescriptions. For the robust approaches (bottom row of Table~\ref{tbl:comparison}), we construct the uncertainty sets from Section~\ref{sec:sro} using the $\ell_1$  norm. For each training sample size, we compute the out-of-sample objective on a test set of size 1000, and we average the results over 100 instances of training data.

In order to select $\epsilon_N$ and other tuning parameters associated with the machine learning weight functions, we first split the data into a training and validation set. We then train the weight functions using the training set, compute decisions for each of the instances in the validation set, and compute the out-of-sample cost on the validation set. We repeat this for a variety of parameter values and select the combination that achieves the best cost on the validation set.
  
  Following a similar reformulation approach as  \citet{esfahani2015data}, we solve the robust approaches \emph{exactly} by observing that
\begin{align*}
\begin{aligned}
&\underset{\bx \in \mathcal{X}, \; \beta \in \R}{\textnormal{minimize}}&& \sum_{i=1}^N w^i_N(\bar{\bgamma}) \sup_{\bzeta \in \mathcal{U}^i_N} \left \{ \beta + \frac{1}{\alpha}\max \{ 0,-\bx^\intercal \bzeta-\beta \} - \lambda \bx^\intercal \bzeta  \right \} \\
=\quad &\underset{\bx \in \mathcal{X}, \; \beta \in \R}{\textnormal{minimize}}&& \sum_{i=1}^N w^i_N(\bar{\bgamma}) \sup_{\bzeta \in \mathcal{U}^i_N} \left \{ \max\left \{\beta - \lambda \bx^\intercal \bzeta ,\; \left(\frac{1}{\alpha} +\lambda \right)\bx^\intercal \bzeta\right \}  \right \} \\
=\quad &\underset{\bx \in \mathcal{X}, \; \beta \in \R}{\textnormal{minimize}}&& \sum_{i=1}^N w^i_N(\bar{\bgamma})  \max\left \{\sup_{\bzeta \in \mathcal{U}^i_N} \left \{\beta - \lambda \bx^\intercal \bzeta \right \} ,\; \sup_{\bzeta \in \mathcal{U}^i_N} \left(\frac{1}{\alpha} +\lambda \right)\bx^\intercal \bzeta  \right \},\\
=\quad &\underset{\bx \in \mathcal{X}, \; \beta \in \R, \bv \in \R^N}{\textnormal{minimize}}&& \sum_{i=1}^N w^i_N(\bar{\bgamma})  v_i \\
&\textnormal{subject to}&& v_i \ge \beta - \lambda \bx^\intercal \bzeta \\
&&& v_i \ge \left(\frac{1}{\alpha} +\lambda \right)\bx^\intercal \bzeta \\
&&& \quad \forall \bzeta \in \mathcal{U}^i_N, \; i \in \{1,\ldots,N \}.
\end{aligned}
\end{align*}
The final expression can be reformulated as a deterministic optimization problem by reformulating the robust constraints.
 
 \begin{figure}[t]
 \FIGURE{
\centering
\includegraphics[width=0.5\textwidth]{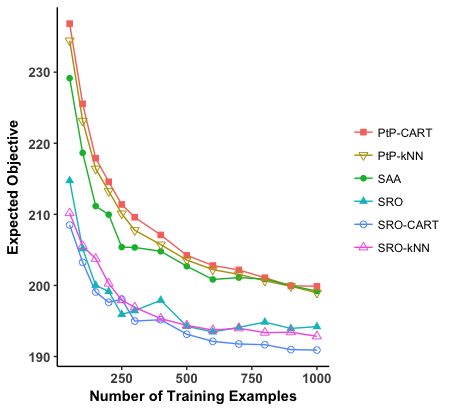}}
{Out-of-sample objective for the portfolio optimization example.\label{fig:portfolio}}
{}
\end{figure}

\paragraph{Results.} In Figure \ref{fig:portfolio}, we show the average out-of-sample objective values using the various methods. 
Consistent with the computational results of \citet{esfahani2015data} and \cite{bertsimas2017bootstrap}, the results underscore the importance of robustness in preventing overfitting and achieving good out-of-sample performance in the small data regime. Indeed, we observe that the sample average approximation, which ignores the auxiliary data, outperforms PtP-$k$NN and PtP-CART when the amount of training data is limited. We believe this is due to the fact the latter methods both throw out training examples, so the methods overfit when the training data is limited, leading to poor out-of-sample performance. 
In contrast, our methods (SRO-$k$NN and SRO-CART) typically achieve the strongest out-of-sample performance, even though the amount of training data is limited.

\subsection{Shipment planning}\label{sec:shipmentplanning}
We finally consider a shipment planning problem in which a decision maker seeks to satisfy demand in several locations from several production facilities while minimizing production and transportation costs. 
 Our problem setting closely follows \cite{bertsimas2014predictive}, in which the decision maker has access to auxiliary side information (promotions, social media, market trends), which may be predictive of future sales in each retail location.

\paragraph{Problem Description.} The decision maker first decides the quantity of inventory $x_f \ge 0$ to produce in each of the production facilities $f \in \mathcal{F} \triangleq \{1,\ldots,| \mathcal{F}| \}$,  at a cost of $p_1$ per unit. The demands $\xi_\ell \ge 0$ in each location $\ell \in \mathcal{L} \triangleq \{1,\ldots, | \mathcal{L} | \}$ are then observed. The decision maker fulfills these demands by shipping $s_{f\ell} \ge 0$ units from facility $f \in \F$ to location $\ell \in \mathcal{L}$ at a per-unit cost of $c_{f\ell} > 0$. Additionally, after observing demand, the decision maker has the opportunity to produce additional units $y_f \ge 0$ in each facility at a cost of $p_2 > p_1$ per unit.  The fulfillment of each unit of demand generates $r > 0$ in revenue. 
Given the above notation and dynamics, the cost incurred by the decision maker is
\begin{align*}
\begin{aligned}
c(\bxi,\bx) = \sum_{f \in \mathcal{F}} p_1 x_f  - \sum_{\ell \in \mathcal{L}} r \xi_\ell \;\;+ \;\;
&\underset{\substack{\bs \in \R^{ \mathcal{L} \times  \mathcal{F}}_+,\;\by \in \R^{ \mathcal{F}}_+}}{\textnormal{minimize}} && \sum_{f \in \F} p_2 y_f + \sum_{f \in \F} \sum_{\ell \in \mathcal{L}} c_{f\ell} s_{f \ell} \\
&\text{subject to} && \sum_{f \in \F} s_{f\ell} \ge \xi_\ell& \forall \ell \in \mathcal{L} \\
&&& \sum_{\ell \in \mathcal{L}} s_{f \ell} \le x_f + y_f& \forall f \in \F.
\end{aligned}
\end{align*}

\paragraph{Experiments.} We perform computational experiments using the same parameters and data generation procedure as \citet{bertsimas2014predictive}. Specifically, we consider an instance with $| \F |=4$, $| \mathcal{L}|=12$, $p_1 = 5$, $p_2 = 100$, and $r = 90$. The network topology,  transportation costs, and the joint distribution of the side information $\bgamma \in \R^3$ and demands $\bxi \in \R^{12}$ are the same as \citet{bertsimas2014predictive}, with the exception that we generate the side information as i.i.d. samples as opposed to an ARMA process (but with the same marginal distribution). 

In our experiments, we compare sample robust optimization with side information, sample average approximation, sample robust optimization, and predictions to prescriptions. For the robust approaches (bottom row of Table~\ref{tbl:comparison}), we construct the uncertainty sets from Section~\ref{sec:sro} using the $\ell_1$ norm and $\Xi = \R^{12}_+$, solve these problems using the multi-policy approximation with linear decision rules described in Section~\ref{sec:approx}, and consider uncertainty sets with radius $\epsilon \in \{100,500\}$. For the approaches using side information (right column of Table~\ref{tbl:comparison}), we used the $k_N$-nearest neighbors with parameter $k_N = \frac{2N}{5}$. 
All solutions were evaluated on a test set of size 100 and the results were averaged over 100 independent training sets. 

 \begin{figure}[t]
 \FIGURE{
\centering
\includegraphics[width=0.5\textwidth]{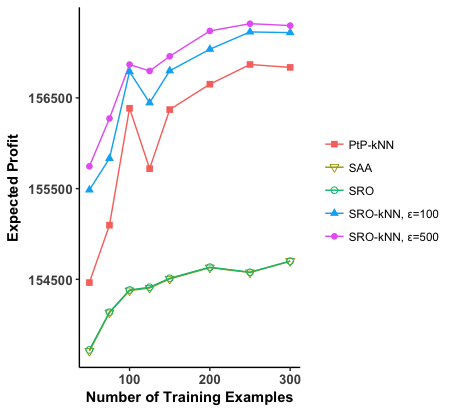}}
{Out-of-sample profit for the shipment planning example.\label{fig:shipment}}
{The profits for SRO and SAA are overlapping. }
\end{figure}
\paragraph{Results.} In Figure \ref{fig:shipment}, we present the average out-of-sample profits of the various methods. 
The results show that the best out-of-sample average profit is attained when using the proposed sample robust optimization with side information. Interestingly, we observe no discernible differences between sample average approximation and sample robust optimization in Figure~\ref{fig:shipment}, suggesting the value gained by incorporating side information in this example. 
Compared to the approach of \cite{bertsimas2014predictive}, sample robust optimization with side information achieves a better out-of-sample average performance for each choice of $\epsilon$. Table~\ref{tbl:shipment_planningpvalue} shows that these differences are statistically significant. This example demonstrates that, in addition to enjoying asymptotic optimality guarantees, sample robust optimization with side information provides meaningful value across various values of $N$.

\begin{table}[t]
\TABLE{Statistical significance for shipment planning problem. \label{tbl:shipment_planningpvalue}}
{\centering
\begin{tabular}{lC{3cm} C{3cm}}
\hline
& \multicolumn{2}{c}{$\epsilon$}\\
  \cline{2-3}\\[-2.5ex]
 $N$& 100 & 500 \\ 
  \hline
50 & $4.6 \times 10^{-13}$ & $5.3\times10^{-16}$\\
75& $1.3\times 10^{-14}$ & $6.4\times10^{-12}$\\
100 & $1.2\times10^{-13}$ & $1.1\times10^{-7}$\\
125 & $2.6\times10^{-15}$ &$1.5\times10^{-11}$\\
150 & $3.4\times10^{-12}$& $1.2\times10^{-6}$\\
 200 & $1.4\times10^{-12}$& $1.0\times10^{-8}$\\
 250 & $3.4\times10^{-10}$& $1.0\times10^{-4}$\\
 300 & $1.8\times10^{-6}$ & $5.2\times10^{-4}$\\
    \hline
\end{tabular}}
{The $p$-values from the Wilcoxon signed rank test for comparison with the predictive to prescriptive analytics method (PtP-$k$NN) and sample robust optimization with side information (SRO-$k$NN).  After adjusting for multiple hypothesis testing, all results are significant at the $\alpha = 0.05$ significance level because all $p$-values are less than $\frac{0.05}{16} \approx 3.1 \times 10^{-3}$.   
}
\end{table}

\section{Conclusion} \label{sec:conclusion} 
In this paper, we introduced \emph{sample robust optimization with side information}, a new approach for solving dynamic optimization problems with side information. Through three computational examples, we demonstrated that our method achieves significantly better out-of-sample performance than scenario-based alternatives. We complemented these empirical observations with theoretical analysis, showing our nonparametric method is asymptotically optimal via a new concentration measure result for local learning methods. Finally, we showed our approach inherits the tractability of robust optimization, scaling to problems with many stages via the multi-policy approximation scheme.

\section*{Acknowledgements}The authors thank the associate editor and two referees for many helpful suggestions that greatly improved the manuscript.

\clearpage
\theendnotes


\bibliographystyle{informs2014} 
\bibliography{DataDriven_MultiStage} 
\ECSwitch


\ECHead{Electronic Companion}

\section{Properties of Weight Functions} \label{appx:weights}

In this section, we show that the $k$-nearest neighbor and kernel regression weight functions satisfy several guarantees. These results are used in the proof of Theorem~\ref{thm:conditionalconcentration}, found in Section~\ref{sec:concentration}. The main result of this section is the following. For convenience, the equations below are numbered the same as in the proof of Theorem~\ref{thm:conditionalconcentration}.
\begin{theorem} \label{thm:weight_functions}
If Assumptions~\ref{ass:params} and \ref{as:aux} hold, then 
\begin{align}
&\text{$\{w^i_N(\bar{\bgamma})\}$ are not functions of $\bxi^1,\ldots,\bxi^N$}; \tag{\ref{line:honesty_assumption}}\\
&\sum_{i=1}^N w^i_N(\bar{\bgamma}) = 1 \text{ and } w^1_N(\bar{\bgamma}),\ldots,w^N_N(\bar{\bgamma}) \ge 0,&\forall N \in \N.  \tag{\ref{line:sum_to_one}}
\end{align}
Moreover, there exists constants $k_2 > 0$ and $\eta > p(2+d_\xi)$ such that 
\begin{align}
&\lim_{N \to \infty} \frac{1}{\epsilon_N}\sum_{i=1}^N w_N^i(\bar{\bgamma}) \|\bgamma{}^i-\bar{\bgamma}\| = 0, &\text{  $\P^\infty$-almost surely}; \tag{\ref{line:lipschitz_convergence}} \\
&\E_{\P^N} \left[\exp\left(\frac{-\theta}{\sum_{i=1}^N w_N^i(\bar{\bgamma})^2}\right)\right] \le \exp(-k_2\theta N^\eta), & \forall \theta \in (0,1), N \in \N \tag{\ref{line:mgf}}. 
\end{align}
\end{theorem}

\begin{proof}{Proof.}
We observe that \eqref{line:honesty_assumption} and \eqref{line:sum_to_one} follow directly from the definitions of the weight functions. The proofs of \eqref{line:lipschitz_convergence} and  \eqref{line:mgf} are split into two parts, one for the $k$-nearest neighbor weights and one for kernel regression weights.

\subsubsection*{k-Nearest Neighbors:}

For the proof of \eqref{line:lipschitz_convergence}, we note
\begin{align*}
\sum_{i=1}^Nw_N^i(\bar{\bgamma}) \|\bgamma^i-\bar{\bgamma}\| \le \|\bgamma^{(k_N)}(\bar{\bgamma})-\bar{\bgamma}\|,
\end{align*}
where $\bgamma^{(k_N)}(\bar{\bgamma})$ denotes the $k_N$th nearest neighbor of $\bar{\bgamma}$ out of $\bgamma^1,\ldots,\bgamma^N$. Therefore, for any $\lambda > 0$,
\begin{align*}
\P^N\left(\sum_{i=1}^Nw_N^i(\bar{\bgamma})\|\bgamma^i-\bar{\bgamma}\| > \lambda\epsilon_N\right) &\le \P^N\left(\|\bgamma^{(k_N)}(\bar{\bgamma})-\bar{\bgamma}\| > \lambda\epsilon_N\right) \\
&\le \P^N\left(\left|\left\{ i : \| \bgamma^i - \bar{\bgamma} \| \le \lambda \epsilon_N \right\}\right| \le k_N-1\right).
\end{align*}
By Assumption \ref{as:aux}, this probability is upper bounded by $\P(\beta \le k_N-1)$, where $\beta \sim \text{Binom}(N,g(\lambda\epsilon_N)^{d_\gamma})$. By Hoeffding's inequality, 
\begin{align*}
\P^N\left(\sum_{i=1}^Nw_N^i(\bar{\bgamma})\|\bgamma^i-\bar{\bgamma}\| > \lambda\epsilon_N\right) &\le \exp\left(\frac{-2(Ng(\lambda k_1/N^p)^{d_\gamma}-k_N+1)^2}{N}\right),
\end{align*}
for $k_N \le Ng(\lambda k_1/N^p)^{d_\gamma} + 1$. We note that this condition on $k_N$ is satisfied for $N$ sufficiently large because $\delta+pd_\gamma < 1$ by Assumption~\ref{ass:params}. Because the right hand side in the above inequality has a finite sum over $N$, \eqref{line:lipschitz_convergence} follows by the Borel Cantelli lemma.

For the proof of \eqref{line:mgf}, it follows from Assumption~\ref{ass:params} that
\begin{equation*}
\sum_{i=1}^Nw_N^i(\bar{\bgamma})^2 \le k_3^{-2} N^{1-2\delta}
\end{equation*}
deterministically (for all sufficiently large $N$ such that $\lceil k_3 N^\delta\rceil \le N-1$) and $2\delta-1 > p(d_\xi + 2)$. Thus, \eqref{line:mgf} follows with $\eta = 2\delta-1$.

\subsubsection*{Kernel regression:} Assumption~\ref{ass:params} stipulates that the kernel function $K(\cdot)$ is Gaussian, triangular, or Epanechnikov, which are defined in Section~\ref{sec:sro}. It is easy to verify that these kernel functions satisfy the following:
\begin{enumerate}
\item $K$ is nonnegative, finite valued, and monotonically decreasing (for nonnegative inputs).
\item $u^\alpha K(u) \to 0$ as $u\to\infty$ for any $\alpha\in\R$.
\item $\exists u^* > 0$ such that $K(u^*) > 0$.
\end{enumerate}
For the proof of \eqref{line:lipschitz_convergence}, define $q > 0$ such that $p<q<\delta$. Letting $D$ be the diameter of $\Gamma$ and $g_N(\bar{\bgamma}) = \sum_{i=1}^NK(\|\bgamma^i-\bar{\bgamma}\|/h_N)$, we have
\begin{align*}
&\sum_{i=1}^N w_N^i(\bar{\bgamma}) \|\bgamma^i-\bar{\bgamma}\| \\
&= \sum_{i=1}^N w_N^i(\bar{\bgamma}) \I\{\|\bgamma^i-\bar{\bgamma}\| \le N^{-q}\}\|\bgamma^i-\bar{\bgamma}\| + \frac{1}{g_N(\bar{\bgamma})}\sum_{i=1}^N K\left(\frac{\|\bgamma^i-\bar{\bgamma}\|}{h_N}\right)\I\{\|\bgamma^i-\bar{\bgamma}\| > N^{-q}\}\|\bgamma^i-\bar{\bgamma}\|\\
&\le N^{-q} + \frac{NDK(N^{-q}/h_N)}{g_N(\bar{\bgamma})},
\end{align*}
where the inequality follows from the monotonicity of $K$. By construction, $N^{-q}/\epsilon_N\to0$, so we just need to handle the second term. We note, for any $\lambda > 0$,
\begin{align*}
&\P^N\left(\frac{NDK(N^{-q}/h_N)}{g_N(\bar{\bgamma})} > \lambda\epsilon_N\right)\le \P^N\left(\sum_{i=1}^NZ_i^N K(u^*) < \frac{NDK(N^{-q}/h_N)}{\lambda\epsilon_N}\right),
\end{align*}
where $Z_i^N = \I\{\|\bgamma^i-\bar{\bgamma}\| \le u^*h_N\}$. To achieve this inequality, we lower bounded each term in $g_N(\bar{\bgamma})$ by $K(u^*)$ or 0, because of the monotonicity of $K$. By Hoeffding's inequality,
\begin{align*}
\P^N\left(\sum_{i=1}^NZ_i^N K(u^*) < \frac{NDK(N^{-q}/h_N)}{\lambda\epsilon_N}\right) 
&\le \exp\left(-\frac{2\left(N\E Z_i^N - \frac{ND}{\lambda\epsilon_NK(u^*)}K(N^{-q}/h_N)\right)_+^2}{N}\right) \\
&\le \exp\left(-\frac{2\left(Ng(u^*h_N)^{d_\gamma} - \frac{ND}{\lambda\epsilon_NK(u^*)}K(N^{-q}/h_N)\right)_+^2}{N}\right) \\
&= \exp\left(-\left(k_5N^{1/2-\delta d_\gamma} - k_6N^{1/2+p}K(k_4N^{-q+\delta})\right)_+^2\right),
\end{align*}
for some constants $k_5,k_6>0$ that do not depend on $N$. We used Assumption \ref{as:aux} for the second inequality. Because $\delta > q$, the second kernel property implies $N^{1/2+p}K(k_4N^{-q+\delta})$ goes to 0 as $N$ goes to infinity, so that term is irrelevant. Because $1/2-\delta d_\gamma > 0$ by Assumption~\ref{ass:params}, the right hand side of the inequality has a finite sum over $N$, and thus \eqref{line:lipschitz_convergence} follows from the Borel Cantelli lemma.

For the proof of \eqref{line:mgf}, define 
\begin{equation*}
v^N = \begin{pmatrix} K(\|\bgamma^1-\bar{\bgamma}\|/h_N) \\ \vdots\\ K(\|\bgamma^N-\bar{\bgamma}\|/h_N)\end{pmatrix}.
\end{equation*}
We note that 
\begin{align*}
\sum_{i=1}^Nw_N^i(\bar{\bgamma})^2 &= \frac{\|v^N\|_2^2}{\|v^N\|_1^2} 
\le \frac{\|v^N\|_\infty}{\|v^N\|_1} 
\le \frac{K(0)}{K(u^*)\sum_{i=1}^N Z_i^N},
\end{align*}
where $Z_i^N$ is defined above. The first inequality follows from Holder's inequality, and the second inequality follows from the monotonicity of $K$. Next, we define $\bar{Z}_i^N$ to be a Bernoulli random variable with parameter $g(u^*h_N)^{d_\gamma}$ for each $i$. For any $\theta\in(0,1)$,
\begin{align*}
\E_{\P^N} \left[\exp\left(\frac{-\theta}{\sum_{i=1}^N w_N^i(\bar{\bgamma})^2}\right)\right] &\le \E_{\P^N} \left[\exp\left(\frac{-\theta K(u^*)\sum_{i=1}^N \bar{Z}_i^N}{K(0)}\right)\right] \\
&= \left(1-g(u^*h_N)^{d_\gamma} + g(u^*h_N)^{d_\gamma}\exp(-\theta K(u^*)/K(0))\right)^N \\
&\le \exp\left(-Ng(u^*h_N)^{d_\gamma}(1-\exp(-\theta K(u^*)/K(0)))\right) \\
&\le \exp\left(-Ng(u^*h_N)^{d_\gamma}\frac{\theta K(u^*)}{2K(0)}\right) \\
&= \exp\left(-\frac{\theta K(u^*)g(k_4u^*)^{d_\gamma}N^{1-\delta d_\gamma}}{2K(0)} \right).
\end{align*}
The first inequality follows because $g(u^* h_N)^{d_\gamma}$ is an upper bound on $\P(\|\bgamma^i-\bar{\bgamma}\| \le u^*h_N)$ by Assumption~\ref{as:aux}. The first equality follows from the definition of the moment generating function for a binomial random variable. The next line follows from the inequality $e^x \ge 1+x$ and the following from the inequality $1-e^{-x} \ge x/2$ for $0\le x\le 1$. Because $1-\delta d_\gamma > p(2+d_\xi)$, this completes the proof of \eqref{line:mgf} with $\eta=1-\delta d_\gamma$ and $k_2=K(u^*)g(k_4u^*)^{d_\gamma}/2K(0)$.
\halmos
\end{proof}

\section{Proof of Theorem~\ref{thm:convergence}} \label{appx:proof_main_theorem}
In this section, we present our proof of Theorem~\ref{thm:convergence}. We make use of the following result from \citet{bertsimas2018multistage} (their {\color{black}Lemma EC.2}), which bounds the difference in worst case objective values between {\color{black}distributionally robust optimizatoin with the type-1 Wasserstein ambiguity set and sample robust optimization\footnote{We make this distinction to update our notation with the latest version of the paper.}} problems. We note that \citet{bertsimas2018multistage} proved the following result for the case that $\Q'$ is the unweighted empirical measure, but their proof carries through for the case here in which $\Q'$ is a weighted empirical measure.
\begin{lemma} \label{lemma:l1_to_was}
Let $\mathcal{Z} \subseteq \R^d$, $f: \mathcal{Z} \to \R$ be measurable, and $\bzeta^1,\ldots,\bzeta^N \in \mathcal{Z}$. Suppose that $$\Q' = \sum_{i=1}^N w^i \delta_{\bzeta^i}$$ for given weights $w^1,\ldots,w^N \ge 0$ that sum to one. If $\theta_2 \ge 2\theta_1 \ge 0$, then
\begin{align*}
\sup_{\Q\in\mathcal{P}(\mathcal{Z}) : \;\mathsf{d}_1(\Q',\Q)\le \theta_1 }\Exp_{\bxi\sim\mathbb{Q}} [f(\bxi)] \le \sum_{i=1}^N w^i \sup_{\bzeta \in \mathcal{Z}: \| \bzeta - \bzeta^i \| \le \theta_2} f(\bzeta)   + {\frac{4\theta_1}{\theta_2}}\sup_{\bzeta \in \mathcal{Z}}| f(\bzeta)|.
\end{align*}
\end{lemma}
We now restate and prove the main result, which combines the new measure concentration result from this paper with similar proof techniques as \citet{bertsimas2018multistage} and \citet{esfahani2015data}.
\begin{repeattheorem}[Theorem~\ref{thm:convergence}.]
Suppose the weight function and uncertainty sets satisfy Assumption~\ref{ass:params}, the joint probability distribution of $(\bgamma,\bxi)$ satisfies Assumptions \ref{as:subgaussian}-\ref{as:aux} from Section~\ref{sec:concentration}, and the cost function satisfies  Assumptions \ref{as:decisionclass_sro}-\ref{as:decisionclass_sp}from Section~\ref{sec:main_result}.
Then, for every $\bar{\bgamma}\in\Gamma$, $$\lim_{N \to \infty} \hat{v}^N(\bar{\bgamma}) = v^*(\bar{\bgamma}), \quad \Prb^\infty\text{-almost surely}.$$ 
\end{repeattheorem}
\begin{proof}{Proof.}
 We break the limit into upper and lower parts. The proof of the lower part follows from an argument similar to that used by \citet{bertsimas2018multistage}. The proof of the upper part follows from the argument used by \citet{esfahani2015data}.
 \subsubsection*{Lower bound.}
We first show that
\begin{align}
\liminf_{N\to\infty} \hat{v}^N(\bar{\bgamma}) 
\ge v^*(\bar{\bgamma}), \quad \Prb^\infty\text{-almost surely}. \label{line:lower_bound}
\end{align}
Indeed, it follows from Assumptions~\ref{ass:params}-\ref{as:subgaussian} and the union bound that there exists $N_0 \in \N$ such that
\begin{align*}
 \P^N\left(\sup_{\bzeta \in \cup_{i=1}^N \mathcal{U}^i_N} \| \bzeta \| > \log N  \right) < \text{exp} (-(\log N)^{1.99}), \quad \forall N \ge N_0. 
\end{align*}
Therefore, the Borel-Cantelli lemma implies that there exists $N_1 \in \N$, $\Prb^\infty$-almost surely, such that 
\begin{align}
\cup_{i=1}^N \mathcal{U}^i_N \subseteq  D_N \triangleq \{\bzeta : \|\bzeta\| \le \log N\}, \quad \forall N \ge N_1. \label{line:upper_bound_on_support}
\end{align}

Consider any $r >0$ such that $\epsilon_N N^{-r}$ satisfies Assumption~\ref{ass:params}, and let $\Pi^N$ denote the set of decision rules which satisfy the conditions of Assumption~\ref{as:decisionclass_sro}. Then, the following holds for all $N \ge N_1 \triangleq \max \{ N_0, 2^{\frac{1}{r}} \}$ and $\bpi \in \Pi^N$:
\begin{align}
&\sup_{\Q \in \mathcal{P}(D_N \cap \Xi): \; \mathsf{d}_1\left(\Q, \hat{\Prb}^N_{\bar{\bgamma}}\right) \le \frac{\epsilon_N}{N^r}}\E_{\bxi\sim\Q}[c^\bpi(\bxi_1,\ldots,\bxi_T)]\notag \\
&{\color{black}\le \sum_{i=1}^N w^N_i(\bar{\bgamma}) \sup_{\bzeta \in D_N \cap \Xi:\; \| \bzeta - \bxi^i\| \le \epsilon_N}c^\bpi(\bzeta_1,\ldots,\bzeta_T) + \frac{4}{N^r} \sup_{\bzeta\in D_N\cap\Xi} \left |c^\bpi(\bzeta_1,\ldots,\bzeta_T) \right|} \notag \\
&= \sum_{i=1}^N w^N_i(\bar{\bgamma}) \sup_{\bzeta \in \mathcal{U}^i_N}c^\bpi(\bzeta_1,\ldots,\bzeta_T) + \frac{4}{N^r} \sup_{\bzeta\in D_N\cap\Xi} \left |c^\bpi(\bzeta_1,\ldots,\bzeta_T) \right| \notag \\
&\le \sum_{i=1}^N w^N_i(\bar{\bgamma}) \sup_{\bzeta \in \mathcal{U}^i_N}c^\bpi(\bzeta_1,\ldots,\bzeta_T) + \frac{4}{N^r}M\left(1 + \max \left \{ \left\| \bzeta \right \|, \sup_{\bzeta' \in \cup_{i=1}^N \mathcal{U}^i_N} \| \bzeta' \| \right \} \right) \notag \\
&\le  \sum_{i=1}^N w^N_i(\bar{\bgamma}) \sup_{\bzeta \in \mathcal{U}^i_N}c^\bpi(\bzeta_1,\ldots,\bzeta_T) +\frac{4M}{N^r} (1+\log N). \label{line:dont_have_good_name}
\end{align}
Indeed, the first inequality follows from Lemma~\ref{lemma:l1_to_was} since $N \ge 2^{\frac{1}{r}}$, the equality follows from $N \ge N_1$, the second inequality holds because $\bpi \in \Pi^N$, and the third and final inequality follows from the definition of $D_N$ and $N \ge N_1$. We observe that the second term in \eqref{line:dont_have_good_name} converges to zero as $N \to \infty$. 

We now observe that
\begin{align*}
\E [c^\bpi(\bxi_1,\ldots,\bxi_T) \mid \bgamma = \bar{\bgamma}]  &\triangleq \E_{\bxi\sim\P_{\bar{\bgamma}}} [c^\bpi(\bxi_1,\ldots,\bxi_T)]  \\
&=\E_{\bxi\sim\P_{\bar{\bgamma}}} [c^\bpi(\bxi_1,\ldots,\bxi_T)\I\{\bxi\notin D_N\}]   +\E_{\bxi\sim\P_{\bar{\bgamma}}} [c^\bpi(\bxi_1,\ldots,\bxi_T)\I\{\bxi\notin D_N\}].
\end{align*}
We handle the first term with the Cauchy-Schwartz inequality,
\begin{equation*}
\E_{\bxi\sim\P_{\bar{\bgamma}}} [c^\bpi(\bxi_1,\ldots,\bxi_T)\I\{\bxi\notin D_N\}]\le \sqrt{\E_{\bxi\sim\P_{\bar{\bgamma}}}[c^\bpi(\bxi_1,\ldots,\bxi_T)^2]\P_{\bar{\bgamma}}(\bxi\notin D_N)}.
\end{equation*}
By Assumption~\ref{as:subgaussian}, the above bound is finite and converges to zero as $N \to \infty$ uniformly over $\bpi \in \Pi^N$. We handle the second term by the new concentration measure from this paper. Specifically, it follows from  Theorem~\ref{thm:conditionalconcentration} that there exists an $N_2 \ge N_1$, $\Prb^\infty$-almost surely, such that $$\mathsf{d}_1(\P_{\bar{\bgamma}},\hat{\P}_{\bar{\bgamma}}^N) \le \frac{\epsilon_N}{N^r}\quad \forall N \ge N_2.$$ Therefore, for all $N \ge N_2$ and decision rules $\bpi \in \Pi^N$:
\begin{align*}
&\E_{\bxi\sim\P_{\bar{\bgamma}}} [c^\bpi(\bxi_1,\ldots,\bxi_T)\I\{\bxi \in D_N\}] \\
&= \E_{\bxi\sim\P_{\bar{\bgamma}}}\left[\left(c^\bpi(\bxi_1,\ldots,\bxi_T)-\inf_{\bzeta \in D_N \cap \Xi}c^\bpi(\bzeta_1,\ldots,\bzeta_T)\right)\I\{\bxi\in D_N\}\right]  +  \underbrace{\P_{\bar{\bgamma}}(\bxi\in D_N) \inf_{\bzeta \in D_N \cap \Xi}c^\bpi(\bzeta_1,\ldots,\bzeta_T)}_{\alpha_N}\\ 
&\le \sup_{\Q \in \mathcal{P}(\Xi): \; \mathsf{d}_1\left(\Q, \hat{\Prb}^N_{\bar{\bgamma}}\right) \le \frac{\epsilon_N}{N^r}} \E_{\bxi\sim \Q}\left[\left(c^\bpi(\bxi_1,\ldots,\bxi_T)-\inf_{\bzeta \in D_N \cap \Xi}c^\bpi(\bzeta_1,\ldots,\bzeta_T)\right)\I\{\bxi\in D_N\}\right] + \alpha_N \\
&= \sup_{\Q \in \mathcal{P}(\Xi \cap D_N): \; \mathsf{d}_1\left(\Q, \hat{\Prb}^N_{\bar{\bgamma}}\right) \le \frac{\epsilon_N}{N^r}} \E_{\bxi\sim \Q}\left[c^\bpi(\bxi_1,\ldots,\bxi_T)-\inf_{\bzeta \in D_N \cap \Xi}c^\bpi(\bzeta_1,\ldots,\bzeta_T)\right] + \alpha_N \\
&= \sup_{\Q \in \mathcal{P}(\Xi \cap D_N): \; \mathsf{d}_1\left(\Q, \hat{\Prb}^N_{\bar{\bgamma}}\right) \le \frac{\epsilon_N}{N^r}} \E_{\bxi\sim\Q}[c^\bpi(\bxi_1,\ldots,\bxi_T)] - \P_{\bar{\bgamma}}(\bxi\notin D_N)  \inf_{\bzeta\in D_N \cap \Xi}  c^\bpi(\bzeta_1,\ldots,\bzeta_T),
\end{align*}
where the inequality follows from $N \ge N_2$. It follows from \eqref{line:upper_bound_on_support} that the second term in the final equality converges to zero as $N \to \infty$ uniformly over $\bpi \in \Pi^N$. 

Combining the above, we conclude that
\begin{align}
\liminf_{N\to\infty} \hat{v}^N(\bar{\bgamma}) &= \liminf_{N \to \infty} \inf_{\bpi \in \Pi}\sum_{i=1}^N w^N_i(\bar{\bgamma}) \sup_{\bzeta \in \mathcal{U}^i_N}c^\bpi(\bzeta_1,\ldots,\bzeta_T) \notag \\
&= \liminf_{N \to \infty} \inf_{\bpi \in \Pi^N}\sum_{i=1}^N w^N_i(\bar{\bgamma}) \sup_{\bzeta \in \mathcal{U}^i_N}c^\bpi(\bzeta_1,\ldots,\bzeta_T) \label{line:applying_restriction_to_pi_N}\\
&\ge \liminf_{N \to \infty} \inf_{\bpi \in \Pi^N} \E [c^\bpi(\bxi_1,\ldots,\bxi_T) \mid \bgamma = \bar{\bgamma}], \quad \Prb^\infty\text{-almost surely} \notag  \\
&\ge \inf_{\bpi \in \Pi} \E [c^\bpi(\bxi_1,\ldots,\bxi_T) \mid \bgamma = \bar{\bgamma}] \label{line:removing_restriction_to_pi_N} \\
&= v^*(\bar{\bgamma}),  \notag
\end{align}
where \eqref{line:applying_restriction_to_pi_N} follows from  Assumption~\ref{as:decisionclass_sro} and \eqref{line:removing_restriction_to_pi_N} follows because $\Pi^N \subseteq \Pi$ for all $N \in \N$. This completes the proof of \eqref{line:lower_bound}. 

 \subsubsection*{Upper bound.}
We now prove that 
\begin{align}
\limsup_{N\to\infty} \hat{v}^N(\bar{\bgamma}) 
\le v^*(\bar{\bgamma}), \quad \Prb^\infty\text{-almost surely}. \label{line:upper_bound}
\end{align}
Indeed, for any arbitrary $\delta > 0$, let $\bpi_\delta \in \Pi$ be a $\delta$-optimal solution for \eqref{eq:main}. Moreover, without any loss of generality, we assume that the decision rule is chosen to satisfy the conditions of Assumption~\ref{as:decisionclass_sp}. Then it follows from \citet[Lemma A.1]{esfahani2015data} that there exists a non-increasing sequence of  functions $f^j(\bzeta_1,\ldots,\bzeta_T)$, $j \in \N$,  such that $$\lim_{j\to\infty}f^j(\bzeta_1,\ldots,\bzeta_T) = c^{\bpi_\delta}(\bzeta_1,\ldots,\bzeta_T), \quad \forall \bzeta \in \Xi$$ and $f^j$ is $L_j$-Lipschitz continuous. Furthermore, for each $N \in \N$, choose any probability distribution $\hat{\Q}^N \in \mathcal{P}(\Xi)$ such that $\mathsf{d}_1(\hat{\Q}^N, \hat{\Prb}^N_{\bar{\bgamma}}) \le \epsilon_N$ and 
$$\sup_{\Q\in \mathcal{P}(\Xi): \; \mathsf{d}_1(\Q, \hat{\Prb}^N_{\bar{\bgamma}}) \le \epsilon_N}\E_{\bxi \sim \Q}[c^{\bpi_\delta}(\bxi_1,\ldots,\bxi_T)] \le \E_{\bxi \sim \hat{\Q}^N}[c^{\bpi_\delta}(\bxi_1,\ldots,\bxi_T)] + \delta.$$ For any $j\in\N$,
\begin{align*}
\limsup_{N\to\infty} \hat{v}^N(\bar{\bgamma}) 
&\le \limsup_{N\to\infty} \sum_{i=1}^N w^N_i(\bar{\bgamma}) \sup_{\bzeta \in \mathcal{U}^i_N}  c^{\bpi_\delta}(\bzeta_1,\ldots,\bzeta_T)  \\
&= \limsup_{N\to\infty} \sup_{\Q\in \mathcal{P}(\Xi): \; \mathsf{d}_\infty(\Q, \hat{\Prb}^N_{\bar{\bgamma}}) \le \epsilon_N}\E_{\bxi \sim \Q}[c^{\bpi_\delta}(\bxi_1,\ldots,\bxi_T)]  \\
&\le \limsup_{N\to\infty} \sup_{\Q\in \mathcal{P}(\Xi): \; \mathsf{d}_1(\Q, \hat{\Prb}^N_{\bar{\bgamma}}) \le \epsilon_N}\E_{\bxi \sim \Q}[c^{\bpi_\delta}(\bxi_1,\ldots,\bxi_T)]  \\
&\le \limsup_{N\to\infty} \E_{\bxi \sim \hat{\Q}^N}[ c^{\bpi_\delta}(\bxi_1,\ldots,\bxi_T)] +\delta\\
&\le \limsup_{N\to\infty} \E_{\bxi \sim \hat{\Q}^N}[ f^j(\bxi_1,\ldots,\bxi_T)] +\delta\\
&\le \limsup_{N\to\infty} \E_{\bxi \sim \P_{\bar{\bgamma}}}[f^j(\bxi_1,\ldots,\bxi_T)] + L_j \mathsf{d}_1(\P_{\bar{\bgamma}},\hat{\Q}^N) + \delta \\
&\le \limsup_{N\to\infty} \E_{\bxi \sim \P_{\bar{\bgamma}}}[f^j(\bxi_1,\ldots,\bxi_T)]  + L_j(\mathsf{d}_1(\P_{\bar{\bgamma}},\hat{\P}^N_{\bar{\bgamma}}) + \mathsf{d}_1(\hat{\Q}^N,\hat{\P}^N_{\bar{\bgamma}}))+\delta \\
&\le \limsup_{N\to\infty}  \E_{\bxi \sim \P_{\bar{\bgamma}}}[f^j(\bxi_1,\ldots,\bxi_T)] + L_j(\mathsf{d}_1(\P_{\bar{\bgamma}},\hat{\P}^N_{\bar{\bgamma}}) + \epsilon_N)+\delta \\
&=  \E_{\P_{\bar{\bgamma}}}[f^j(\bxi_1,\ldots,\bxi_T)] +\delta, \;\;\;\text{$\P^\infty$-almost surely},
\end{align*}
where we have used the relationship between sample robust optimization and distributionally robust optimization with the type-$\infty$ Wasserstein ambiguity set for the first equality \cite[Section 6]{bertsimas2018multistage}, the  fact $\mathsf{d}_1(\P,\Q) \le \mathsf{d}_\infty(\P,\Q)$ for the second inequality, the dual form of the 1-Wasserstein metric for the fifth inequality (because $f^j$ is $L_j$-Lipschitz), and Theorem \ref{thm:conditionalconcentration} for the equality. Taking the limit as $j\to\infty$, and applying the monotone convergence theorem (which is allowed because $\E_{\bxi \sim \P_{\bar{\bgamma}}}|f^1(\bxi_1,\ldots,\bxi_T)| \le L_1\E_{\bxi \sim \P_{\bar{\bgamma}}}\|\bxi\| + |f^1(0)| < \infty$ by Assumption \ref{as:aux}), gives
\begin{equation*}
\limsup_{N\to\infty}\hat{v}^N(\bar{\bgamma}) \le\E_{\bxi \sim \Prb_{\bar{\bgamma}}}[c^{\bpi_\delta}(\bxi_1,\ldots,\bxi_T)]  + \delta \le v^*(\bar{\bgamma}) + 2\delta, \quad \P^\infty\textnormal{-almost surely}.
\end{equation*}
Since $\delta > 0$ was chosen arbitrarily, the proof of \eqref{line:upper_bound} is complete. \halmos
\end{proof}

\section{Proof of Theorem~\ref{thm:multi_policy}} \label{appx:approx}
In this section, we present our proof of Theorem~\ref{thm:multi_policy} from Section~\ref{sec:approx}. We restate the theorem here for convenience. 
\begin{repeattheorem}[Theorem~\ref{thm:multi_policy}.]
For cost functions of the form \eqref{eq:costfunction}, $\tilde{v}^N(\bar{\bgamma}) = \hat{v}^N(\bar{\bgamma})$. 
\end{repeattheorem}
\begin{proof}{Proof.}
We first show that $\tilde{v}^N(\bar{\bgamma}) \ge \hat{v}^N(\bar{\bgamma})$. Indeed, consider any primary decision rule $\bar{\bpi}$ and auxiliary decision rules $\bar{\by}_1^i,\ldots,\bar{\by}_T^i$ for each $i \in \{1,\ldots,N \}$ which are optimal for \eqref{prob:multi_policy}.\footnote{If no optimal solution exists, then we may choose any $\eta$-optimal solution.}  Then, it follows from feasibility to  \eqref{prob:multi_policy} that
\begin{align*}
\bh^\intercal_t \bar{\by}_t^i(\bzeta_1,\ldots,\bzeta_t) \ge \min_{\by_t \in \R^{d_y^t}} \left \{\bh^\intercal_t \by_t: \;  \sum_{s=1}^t \bba_{t,s} \bar{\bpi}_s(\bzeta_1,\ldots,\bzeta_{s-1}) + \sum_{s=1}^t \bbb_{t,s} \bzeta_s + \bbc_t \by_t \le \bd_t  \right \}
\end{align*}
for each $i \in \{1,\ldots,N\}$, $\bzeta \in \mathcal{U}^i_N$, and $t \in \{1,\ldots,T\}$. Thus, 
\begin{align*}
\hat{v}^N(\bar{\bgamma}) &= \min_{\bpi \in \Pi} \sum_{i=1}^N w^i_N(\bar{\bgamma}) c^{\bpi}(\bzeta_1,\ldots,\bzeta_T) \\
&\le \sum_{i=1}^N w^i_N(\bar{\bgamma}) c^{\bar{\bpi}}(\bzeta_1,\ldots,\bzeta_T) \\
&\le \sum_{i=1}^N w^i_N(\bar{\bgamma})  \sup_{\bzeta \in \mathcal{U}^i_N} \sum_{t=1}^T \left({\bf f}_t^\intercal \bar{\bpi}_t(\bzeta_1,\ldots,\bzeta_{t-1}) + \bg_t^\intercal \bzeta_t + \bh_t^\intercal \bar{\by}_t^i(\bzeta_1,\ldots,\bzeta_{t}) \right) = \tilde{v}^N(\bar{\bgamma}).
\end{align*}
The other side of the inequality follows from similar reasoning. Indeed, let $\bar{\bpi}$ be an optimal solution to \eqref{eq:sro}. For each $i \in \{1,\ldots,N\}$ and $t \in \{1,\ldots,T\}$, define $\bar{\by}^i_t \in \mathcal{R}_t$ as any decision rule that satisfies
\begin{align*}
\bar{\by}^i_t(\bzeta_1,\ldots,\bzeta_t) \in \argmin_{\by_t \in \R^{d_y^t}} \left \{\bh^\intercal_t \by_t: \;  \sum_{s=1}^t \bba_{t,s} \bar{\bpi}_s(\bzeta_1,\ldots,\bzeta_{s-1}) + \sum_{s=1}^t \bbb_{t,s} \bzeta_s + \bbc_t \by_t \le \bd_t  \right \}
\end{align*}
for every $\bzeta \in \mathcal{U}^i_N$. Then,
\begin{align*}
\tilde{v}^N(\bar{\bgamma}) &\le \sum_{i=1}^N w^i_N(\bar{\bgamma})  \sup_{\bzeta \in \mathcal{U}^i_N} \sum_{t=1}^T \left({\bf f}_t^\intercal \bar{\bpi}_t(\bzeta_1,\ldots,\bzeta_{t-1}) + \bg_t^\intercal \bzeta_t + \bh_t^\intercal \bar{\by}_t^i(\bzeta_1,\ldots,\bzeta_{t}) \right) \\
&=  \sum_{i=1}^N w^i_N(\bar{\bgamma})  \sup_{\bzeta \in \mathcal{U}^i_N} c^{\bar{\bpi}}(\bzeta_1,\ldots,\bzeta_T) = \hat{v}^N(\bar{\bgamma}).
\end{align*}
Combining the above inequalities, the proof is complete. 
\Halmos \end{proof}

\section{Tractable Reformulation of the Multi-Policy Approximation} \label{appx:reformulation}
For completeness, we now show how to reformulate the multi-policy approximation scheme with linear decision rules from Section~\ref{sec:approx} into a deterministic optimization problem using standard techniques from robust optimization.

We begin by transforming \eqref{prob:multi_policy} with linear decision rules into a more compact representation. First, we combine the primary linear decision rules across stages as
\begin{align*}
\bx_0 &= \begin{bmatrix}
\bx_{1,0} \\
\vdots \\
\bx_{T,0}
\end{bmatrix} \in \R^{d_x},  &
\bbx &= \begin{bmatrix}
\bzero &\bzero & \bzero & \cdots & \bzero & \bzero & \bzero\\
\bbx_{2,1} & \bzero & \bzero & \cdots& \bzero & \bzero & \bzero \\
\bbx_{3,1} & \bbx_{3,2} & \bzero & \cdots& \bzero & \bzero & \bzero\\
\vdots & \vdots &\vdots &  \ddots & \vdots & \vdots & \vdots \\
\bbx_{T-2,1} & \bbx_{T-2,2} & \bbx_{T-2,3}&\cdots & \bzero& \bzero & \bzero\\
\bbx_{T-1,1} & \bbx_{T-1,2} & \bbx_{T-1,3}&\cdots & \bbx_{T-1,T-2} & \bzero & \bzero\\
\bbx_{T,1} & \bbx_{T,2} & \bbx_{T,3}&\cdots & \bbx_{T,T-2} & \bbx_{T,T-1} & \bzero\\
\end{bmatrix} \in \R^{d_x \times d_\xi}.
\end{align*}
We note that the zero entries in the above matrix are necessary to ensure that the linear decision rules are non-anticipative. Similarly, for each $i \in \{1,\ldots,N\}$, we represent the auxiliary linear decision rules as
\begin{align*}
\by_{0}^i &= \begin{bmatrix}
\by_{1,0}^i \\
 \vdots \\
\by_{T,0}^i 
\end{bmatrix} \in \R^{d_y}, &\bby^i &= \begin{bmatrix}
\bby^i_{1,1} & \bzero &  \cdots& \bzero & \bzero  \\
\bby^i_{2,1} & \bby^i_{2,2} & \cdots& \bzero & \bzero \\
\vdots & \vdots  &  \ddots & \vdots & \vdots \\
\bby^i_{T-1,1} & \bby^i_{T-1,2}&\cdots & \bby^i_{T-1,T-1} & \bzero \\
\bby^i_{T,1} & \bby^i_{T,2} &\cdots & \bby^i_{t,t-1} & \bby^i_{T,T} \\
\end{bmatrix} \in \R^{d_y \times d_\xi}.
\end{align*}
We now combine the problem parameters. Let $\bd = (\bd_1,\ldots,\bd_T) \in \R^{m}$ and
\begin{align*}
{\bf f} &= \begin{bmatrix}
{\bf f}_1 \\
\vdots \\
{\bf f}_T
\end{bmatrix} \in \R^{d_x}, &\bba &=  \begin{bmatrix}
\bba_{1,1} & \bzero &  \cdots& \bzero & \bzero  \\
\bba_{2,1} & \bba_{2,2} & \cdots& \bzero & \bzero \\
\vdots & \vdots  &  \ddots & \vdots & \vdots \\
\bba_{T-1,1} & \bba_{T-1,2}&\cdots & \bba_{T-1,T-1} & \bzero \\
\bba_{T,1} & \bba_{T,2} &\cdots & \bba_{t,t-1} & \bba_{T,T} \\
\end{bmatrix} \in \R^{m \times d_x}, \\
\bg &= \begin{bmatrix}
\bg_1 \\
\vdots \\
\bg_T
\end{bmatrix} \in \R^{d_\xi}, &\bbb &=  \begin{bmatrix}
\bbb_{1,1} & \bzero &  \cdots& \bzero & \bzero  \\
\bbb_{2,1} & \bbb_{2,2} & \cdots& \bzero & \bzero \\
\vdots & \vdots  &  \ddots & \vdots & \vdots \\
\bbb_{T-1,1} & \bbb_{T-1,2}&\cdots & \bbb_{T-1,T-1} & \bzero \\
\bbb_{T,1} & \bbb_{T,2} &\cdots & \bbb_{t,t-1} & \bbb_{T,T} \\
\end{bmatrix} \in \R^{m \times d_x},\\
\bh &= \begin{bmatrix}
\bh_1 \\
\vdots \\
\bh_T
\end{bmatrix} \in \R^{d_y}, &\bbc &=  \begin{bmatrix}
\bbc_{1,1} & \bzero &  \cdots& \bzero & \bzero  \\
\bzero & \bbc_{2,2} & \cdots& \bzero & \bzero \\
\vdots & \vdots  &  \ddots & \vdots & \vdots \\
\bzero& \bzero&\cdots & \bbc_{T-1,T-1} & \bzero \\
\bzero& \bzero&\cdots & \bzero & \bbc_{T,T} \\
\end{bmatrix} \in \R^{m \times d_x}.
\end{align*}
Therefore, using the above compact notation, we can rewrite the multi-policy approximation with linear decision rules as
\begin{equation} \label{prob:multi_policy_compact}
\begin{aligned}
&\underset{\substack{\bx_0 \in \R^{d_x}, \bbx \in \R^{d_{x} \times d_\xi}\\
 \by_0^i \in \R^{d_y},\; \bby^i  \in \R^{d_y \times d_\xi} }}{\textnormal{minimize}} && \sum_{i=1}^N w^i_N(\bar{\bgamma}) \sup_{\bzeta \in \mathcal{U}^i_N} \left \{ {\bf f}^\intercal  (\bx_0 + \bbx \bzeta) + \bg^\intercal \bzeta + \bh^\intercal \left(\by_0^i + \bby^i \bzeta \right) \right \}\\
&\textnormal{subject to}&&  \bba(\bx_0 + \bbx \bzeta)  + \bbb \bzeta + \bbc\left(\by_0^i + \bby^i \bzeta \right) \le \bd  \\
&&& \bx_0 + \bbx \bzeta \in \mathcal{X} \\
&&&\quad \forall \bzeta \in \mathcal{U}_N^i, \;i \in \{1,\ldots,N\},
\end{aligned}
\end{equation}
where  $\mathcal{X} \triangleq \mathcal{X}_1 \times \cdots \times \mathcal{X}_T$ and the matrices $\bbx$ and $\bby$ are non-anticipative.   
 Note that the linear decision rules in the above optimization problem are represented using $O(d_\xi \max \{ d_x,  N d_y \})$ decision variables, where $d_x \triangleq d_x^1 + \cdots + d_x^T$ and $d_y \triangleq d_y^1 + \cdots + d_y^T$. Thus, the complexity of representing the primary and auxiliary linear decision rules scales efficiently both in the size of the dataset and the number of stages.  
For simplicity, we present the reformulation for the case in which there are no constraints on the decision variables and nonnegativity constraints on the random variables.
\begin{theorem}\label{thm:reformulation}
Suppose $\Xi = \R^{d_\xi}_+$ and $\mathcal{X} = \R^{d_x}$. Then, \eqref{prob:multi_policy_compact} is equivalent to
\begin{equation*}
\begin{aligned}
&\underset{\substack{\bx_0 \in \R^{d_x}, \bbx \in \R^{d_{x} \times d_\xi}\\
 \by_0^i \in \R^{d_y},\; \bby^i  \in \R^{d_y \times d_\xi} \\
 \bLambda^i \in \R^{m \times d_\xi}_+, \;  \bs^i \in \R^{d_\xi}_+}}{\textnormal{minimize}}&&
 \sum_{i=1}^Nw^i_N(\bar{\bgamma}) \left( {\bf f}^\intercal \left(\bx_0 + \bbx \bxi^i \right) + \bg^\intercal \bxi^i + \bh^\intercal \left(\by_0^i + \bby^i \bxi^i \right) +  (\bs^i)^\intercal \bxi^i + \epsilon_N \left\|  \bbx^\intercal {\bf f} + \bg + (\bby^i)^\intercal \bh  + \bs^i \right\|_*  \right)\\
&\textnormal{subject to}&& \bba \left(\bx_0 + \bbx \bxi^i \right) +  \bbb \bxi^i + \bbc \left(\by_0^i + \bby^i \bxi^i \right)  + \bLambda^i \bxi{}^i + \epsilon_N \left\| \bba \bbx  + \bbb + \bbc \bby^i + \bLambda^i  \right \|_*  \le \bd\\
&&& \quad \forall i \in \{1,\ldots,N\}.
\end{aligned}
\end{equation*}
where  $\| \bbz \|_* \triangleq (
\| {\bf z}_1 \|_*, \ldots,
\|{\bf z}_r \|_*) \in \R^r$ for any matrix $\bbz \in \R^{r \times n}$.
\end{theorem}
\begin{proof}{Proof. }
For any $\bc \in \R^{d_\xi}$ and $\bxi \in \Xi$, it follows directly from strong duality for conic optimization that
\begin{align*}
\max_{\bzeta \ge \bzero} \left \{ \bc^\intercal \bzeta: \; \| \bzeta - \bxi{} \| \le \epsilon \right \} = \min_{\blambda \ge \bzero} \left \{ (\bc + \blambda)^\intercal \bxi{} + \epsilon \left \| \bc + \blambda \right \|_* \right \}.
\end{align*}
We use this result to reformulate the objective and constraints of \eqref{prob:multi_policy_compact}. First, let the $j$-th rows of $\bba,\bbb,\bbc$ and the $j$-th element of $\bd$ be denoted by $\ba_j \in \R^{d_x}$, $\bb_j \in \R^{\xi}$, $\bc_j \in \R^{d_y},$ and $d_j \in \R$.  Then, each robust constraint has the form
\begin{align*}
 \ba_{j}^\intercal (\bx_0 + \bbx \bzeta) + \bb_j^\intercal \bzeta + \bc_j^\intercal (\by_0^i + \bby^i \bzeta) \le d_j \quad \forall \bzeta \in \mathcal{U}^i_N.
\end{align*}
Rearranging terms,
\begin{align*}
( \ba_{j}^\intercal  \bbx + \bb_j^\intercal + \bc_j^\intercal \bby^i)\bzeta \le d_j - \ba_j^\intercal \bx_0 - \bc_j^\intercal \by_0^i \quad \forall \bzeta \in \mathcal{U}^i_N,
\end{align*}
which applying duality becomes
\begin{align*}
\exists \blambda^{i}_j \ge \bzero: \; \left(\bbx^\intercal \ba_j  + \bb_j + (\bby^i)^\intercal \bc_j  + \blambda_{j}^i \right)^\intercal \bxi{}^i + \epsilon_N \left\| \bbx^\intercal \ba_j  + \bb_j + (\bby^i)^\intercal \bc_j  + \blambda_{j}^i  \right \|_*  \le d_j - \ba_j^\intercal \bx_0 - \bc_j^\intercal \by^i_0.
\end{align*}
Rearranging terms, the robust constraints for each $i \in \{1,\ldots,N\}$ are satisfied if and only if
\begin{align*}
\exists \bLambda^{i} \ge \bzero: \; \bba \left(\bx_0 + \bbx \bxi^i \right) +  \bbb \bxi^i + \bbc \left(\by_0^i + \bby^i \bxi^i \right)  + \bLambda^i \bxi{}^i + \epsilon_N \left\| \bba \bbx  + \bbb + \bbc \bby^i + \bLambda^i  \right \|_*  \le \bd,
\end{align*}
where the dual norm for a matrix is applied separately for each row.
Similarly, the objective function takes the form
\begin{align*}
& \sum_{i=1}^Nw^i_N(\bar{\bgamma})  \sup_{\bzeta \in \mathcal{U}^i_N} \left \{ {\bf f}^\intercal  (\bx_0 + \bbx \bzeta) + \bg^\intercal \bzeta + \bh^\intercal \left(\by_0^i + \bby^i \bzeta \right) \right \}\\
&= \sum_{i=1}^Nw^i_N(\bar{\bgamma}) \left({\bf f}^\intercal \bx_0 + \bh^\intercal \by_0^i + \sup_{\bzeta \in \mathcal{U}^i_N}  \left( {\bf f}^\intercal  \bbx + \bg^\intercal + \bh^\intercal \bby^i \right) \bzeta  \right) \\
&= \sum_{i=1}^Nw^i_N(\bar{\bgamma}) \left({\bf f}^\intercal \bx_0 + \bh^\intercal \by_0^i + \inf_{\bs^i \ge \bzero}\left \{ \left(   \bbx^\intercal {\bf f} + \bg + (\bby^i)^\intercal \bh  + \bs^i \right)^\intercal \bxi^i + \epsilon_N \left\|  \bbx^\intercal {\bf f} + \bg + (\bby^i)^\intercal \bh  + \bs^i \right\|_* \right)  \right \}\\
&= \sum_{i=1}^Nw^i_N(\bar{\bgamma}) \left( {\bf f}^\intercal \left(\bx_0 + \bbx \bxi^i \right) + \bg^\intercal \bxi^i + \bh^\intercal \left(\by_0^i + \bby^i \bxi^i \right) + \inf_{\bs^i \ge \bzero} \left \{     (\bs^i)^\intercal \bxi^i + \epsilon_N \left\|  \bbx^\intercal {\bf f} + \bg + (\bby^i)^\intercal \bh  + \bs^i \right\|_*  \right \} \right).
\end{align*}
Combining the reformulations above, we obtain the desired reformulation.
\Halmos \end{proof}

\end{document}